\documentclass[12pt,reqno]{amsart}

\usepackage{amsthm}
\usepackage{amscd}
\usepackage{amsfonts}
\usepackage{amssymb}
\usepackage{amsgen}
\usepackage{amsmath, bm}
\usepackage{amsopn}
\usepackage{verbatim}
\usepackage{xypic}
\usepackage{xspace}
\usepackage{multicol}
\usepackage{url}
\usepackage{upref}
\usepackage{pgf}
\usepackage{tikz}
\usepackage[OT2,T1]{fontenc}
\DeclareSymbolFont{cyrletters}{OT2}{wncyr}{m}{n}
\DeclareMathSymbol{\Sha}{\mathalpha}{cyrletters}{"58}

\theoremstyle{plain}
\newtheorem{thm}{Theorem}[section]
\newtheorem{lem}[thm]{Lemma}
\newtheorem{prop}[thm]{Proposition}
\newtheorem{cor}[thm]{Corollary}
\newtheorem{conj}[thm]{Conjecture}

\theoremstyle{definition}

\newtheorem{defn}[thm]{Definition}

\theoremstyle{remark}
\newtheorem{rem}[thm]{Remark}

\newtheorem{eg}[thm]{Example}

 \def\ppmod#1{\, {\rm (mod\, }{#1)}}

 \font\cyr=wncyr10
 \newcommand{\nc}{\newcommand}

\DeclareMathOperator{\Li}{{Li}}

\DeclareMathOperator{\dep}{{dp}}

\DeclareMathOperator{\tLi}{{\widetilde{Li}}}
\DeclareMathOperator{\MD}{\mathcal{MD}}
\DeclareMathOperator{\qMZ}{q\mathcal{M}\mathcal{Z}}

\DeclareMathOperator{\MZ}{\mathcal{MZ}}

\nc{\per}[1]{\underset{#1}{\boldsymbol \pi}\,}
\nc{\sha}{{\mbox{\cyr x}}}

 \nc{\baro}{\overset{\rm o}}

 \nc{\MT}{{\rm MT}}
 \nc{\XX}{{X}}
 \nc{\gF}{{\varPhi}}
 \nc{\ot}{\otimes}

 \nc{\wht}{\widehat}
 \nc{\bwg}{{\bigwedge}}
 \nc{\wg}{{\wedge}}
 \nc{\mmu}{{\boldsymbol{\mu}}}
 \nc{\mal}{{{\scriptstyle \maltese}}}
 \nc{\fA}{{\mathfrak A}}
 \nc{\HH}{{\mathfrak H}}
 \nc{\ra}{\rightarrow}
 \nc{\ors}{{\bfs}}
 \nc{\orr}{{\bfr}}
 \nc{\os}{{\overset}}
 \nc{\G}{{\mathbb G}}
 \nc{\F}{{\mathbb F}}
 \nc{\HHH}{{\mathbb H}}
 \nc{\Z}{{\mathbb Z}}
 \nc{\R}{{\mathbb R}}
 \nc{\N}{{\mathbb N}}
 \nc{\ZN}{{\mathbb Z_{\ge 0}}}
 \nc{\Q}{{\mathbb Q}}
 \nc{\C}{{\mathbb C}}
 \nc{\CP}{{\mathbb{CP}}}
 \nc{\Cnn}{{\mathbb C}_{\ge 0}}
 \nc{\Cp}{{\mathbb C}_{>0}}
 \nc{\MPV}{{\mathcal{MPV}}}
 
 \nc{\tB}{{\tilde B}}
 \nc{\tg}{{\tilde{g}}}
 \nc{\tn}{{\tilde{n}}}
 \nc{\tgz}{{\tilde{\zeta}}}

 \nc{\tPsi}{{\tilde{\Psi}}}

 \nc{\suf}{{\ast\,}}
 \nc{\sufq}{{\ast_q\,}}
 \nc{\gam}{{\gamma}}
 \nc{\gG}{{\Gamma}}
 \nc{\om}{{\omega}}
 \nc{\vep}{{\varepsilon}}
 \nc{\ga}{{\alpha}}
 \nc{\gl}{{\lambda}}
 \nc{\gb}{{\beta}}
 \nc{\gf}{{\varphi}}
 \nc{\gd}{{\delta}}
 \nc{\orgd}{{\vec \gd\,}}
 \nc{\gs}{{\sigma}}
 \nc{\gth}{{\theta}}
 \nc{\gS}{{\Sigma}}
 \nc{\gk}{{\kappa}}
  \nc{\gz}{{\zeta}}
 \nc{\gO}{{\Omega}}
 \nc{\sif}{{\mathcal S}}
 \nc{\gt}{{\tau}}
 \nc{\Lra}{\Longrightarrow}
 \nc{\lra}{\longrightarrow}
 \nc{\lmaps}{\longmapsto}
 \nc{\fS}{{\mathfrak S}}
 \nc{\DD}{{\mathfrak D}}
 \nc{\Llra}{\Longleftrightarrow}
 \nc{\ol}{\overline}
 \nc{\ola}{\overleftarrow}
 \nc{\lms}{\longmapsto}
 \nc{\cv}{{{\mathsf c}{\mathsf v}}}
 \nc{\zq}{{\zeta_q}}
 \nc\qup{{q\uparrow 1}}
 \nc{\us}{\underset}
 \nc{\gD}{{\Delta}}
 \nc{\bi}{{\bf i}}
 \nc{\bfone}{{\bf 1}}
 \nc{\bfa}{{\bf a}}
 \nc{\bfb}{{\bf b}}
 \nc{\bfc}{{\bf c}}
 \nc{\bfd}{{\bf d}}
 \nc{\bfe}{{\bf e}}
 \nc{\bff}{{\bf f}}
 \nc{\bfg}{{\bf g}}
 \nc{\bfi}{{\bf i}}
 \nc{\bfj}{{\bf j}}

 \nc{\bfn}{{\bf n}}
 \nc{\bfl}{{\bf l}}
 \nc{\bfk}{{\bf k}}
 \nc{\bfm}{{\bf m}}
 \nc{\bfo}{{\bf o}}
 \nc{\bfp}{{\bf p}}
 \nc{\bfq}{{\bf q}}
 \nc{\bfr}{{\bf r}}
 \nc{\bfs}{{\bf s}}
 \nc{\bfS}{{\bf S}}
 \nc{\bft}{{\bf t}}
 \nc{\bfu}{{\bf u}}
 \nc{\bfv}{{\bf v}}
 \nc{\bfw}{{\bf w}}
 \nc{\bfx}{{\bf x}}
 \nc{\bfy}{{\bf y}}
 \nc{\bfz}{{\bf z}}
 \nc{\bfB}{{\bf B}}
 \nc{\bfP}{{\bf P}}
 \nc{\bfQ}{{\bf Q}}
 \nc{\bfY}{{\bf Y}}
\nc{\bfgb}{{\bm \gb}}
 \nc{\bfga}{{\bm \ga}}
 \nc{\bfrho}{{\boldsymbol \rho}}
 \nc{\bfchi}{{\boldsymbol \chi}}
 \nc{\QX}{{\Q\langle \bfX\rangle}}
 \nc{\QY}{{\Q\langle \bfY\rangle}}
 \nc{\CX}{{\C\langle \bfX\rangle}}
 \nc{\CY}{{\C\langle \bfY\rangle}}
 \nc{\QXX}{{\Q\langle\!\langle \bfX\rangle\!\rangle}}
 \nc{\QYY}{{\Q\langle\!\langle \bfY\rangle\!\rangle}}
 \nc{\CXX}{{\C\langle\!\langle \bfX\rangle\!\rangle}}
 \nc{\CYY}{{\C\langle\!\langle \bfY\rangle\!\rangle}}

 \nc{\bbA}{{\mathbb A}}
 \nc{\bbB}{{\mathbb B}}
 \nc{\bbC}{{\mathbb C}}
 \nc{\bbD}{{\mathbb D}}
 \nc{\bbE}{{\mathbb E}}
 \nc{\bbF}{{\mathbb F}}
 \nc{\bbG}{{\mathbb G}}
 \nc{\bbH}{{\mathbb H}}
 \nc{\bbI}{{\mathbb I}}
 \nc{\bbJ}{{\mathbb J}}
 \nc{\bbK}{{\mathbb K}}
 \nc{\bbL}{{\mathbb L}}
 \nc{\bbM}{{\mathbb M}}
 \nc{\bbN}{{\mathbb N}}
 \nc{\bbO}{{\mathbb O}}
 \nc{\bbP}{{\mathbb P}}
 \nc{\bbQ}{{\mathbb Q}}
 \nc{\bbR}{{\mathbb R}}
 \nc{\bbS}{{\mathbb S}}
 \nc{\bbT}{{\mathbb T}}
 \nc{\bbU}{{\mathbb U}}
 \nc{\bbV}{{\mathbb V}}
 \nc{\bbW}{{\mathbb W}}
 \nc{\bbX}{{\mathbb X}}
 \nc{\bbY}{{\mathbb Y}}
 \nc{\bbZ}{{\mathbb Z}}
 \nc{\bba}{{\mathbb a}}
 \nc{\bbb}{{\mathbb b}}
 \nc{\bbc}{{\mathbb c}}
 \nc{\bbd}{{\mathbb d}}
 \nc{\bbe}{{\mathbb e}}
 \nc{\bbf}{{\mathbb f}}
 \nc{\bbg}{{\mathbb g}}
 \nc{\bbh}{{\mathbb h}}
 \nc{\bbi}{{\mathbb i}}
 \nc{\bbk}{{\mathbb k}}
 \nc{\bbl}{{\mathbb l}}
 \nc{\bbm}{{\mathbb m}}
 \nc{\bbn}{{\mathbb n}}
 \nc{\bbo}{{\mathbb o}}
 \nc{\bbp}{{\mathbb p}}
 \nc{\bbq}{{\mathbb q}}
 \nc{\bbr}{{\mathbb r}}
 \nc{\bbs}{{\mathbb s}}
 \nc{\bbt}{{\mathbb t}}
 \nc{\bbu}{{\mathbb u}}
 \nc{\bbv}{{\mathbb v}}
 \nc{\bbw}{{\mathbb w}}
 \nc{\bbx}{{\mathbb x}}
 \nc{\bby}{{\mathbb y}}
 \nc{\bbz}{{\mathbb z}}

 \nc{\MZV}{{\mathcal{MZV}}}
 \nc{\calA}{{\mathcal A}}
 \nc{\calB}{{\mathcal B}}
 \nc{\calC}{{\mathcal C}}
 \nc{\calD}{{\mathcal D}}
 \nc{\calE}{{\mathcal E}}
 \nc{\calF}{{\mathcal F}}
 \nc{\calG}{{\mathcal G}}
 \nc{\calH}{{\mathcal H}}
 \nc{\calI}{{\mathcal I}}
 \nc{\calJ}{{\mathcal J}}
 \nc{\calK}{{\mathcal K}}
 \nc{\calL}{{\mathcal L}}
 \nc{\calM}{{\mathcal M}}
 \nc{\calN}{{\mathcal N}}
 \nc{\calO}{{\mathcal O}}
 \nc{\calP}{{\mathcal P}}
 \nc{\calQ}{{\mathcal Q}}
 \nc{\calR}{{\mathcal R}}
 \nc{\calS}{{\mathcal S}}
 \nc{\calT}{{\mathcal T}}
 \nc{\calU}{{\mathcal U}}
 \nc{\calV}{{\mathcal V}}
 \nc{\calW}{{\mathcal W}}
 \nc{\calX}{{\mathcal X}}
 \nc{\calY}{{\mathcal Y}}
 \nc{\calZ}{{\mathcal Z}}
  \nc{\cala}{{\mathcal a}}
 \nc{\calb}{{\mathcal b}}
 \nc{\calc}{{\mathcal c}}
 \nc{\cald}{{\mathcal d}}
 \nc{\cale}{{\mathcal e}}
 \nc{\calf}{{\mathcal f}}
 \nc{\calg}{{\mathcal g}}
 \nc{\calh}{{\mathcal h}}
 \nc{\cali}{{\mathcal i}}
 \nc{\calj}{{\mathcal j}}
 \nc{\calk}{{\mathcal k}}
 \nc{\call}{{\mathcal l}}
 \nc{\calm}{{\mathcal m}}
 \nc{\caln}{{\mathcal n}}
 \nc{\calo}{{\mathcal o}}
 \nc{\calp}{{\mathsf p}}
 \nc{\calq}{{\mathcal q}}
 \nc{\calr}{{\mathcal r}}
 \nc{\cals}{{\mathcal s}}
 \nc{\calt}{{\mathcal t}}
 \nc{\calu}{{\mathcal u}}
 \nc{\calv}{{\mathcal v}}
 \nc{\calw}{{\mathcal w}}
 \nc{\calx}{{\mathcal x}}
 \nc{\caly}{{\mathcal y}}
 \nc{\calz}{{\mathcal z}}

 \nc{\frakA}{{\mathfrak A}}
 \nc{\frakB}{{\mathfrak B}}
 \nc{\frakC}{{\mathfrak C}}
 \nc{\frakD}{{\mathfrak D}}
 \nc{\frakE}{{\mathfrak E}}
 \nc{\frakF}{{\mathfrak F}}
 \nc{\frakG}{{\mathfrak G}}
 \nc{\frakH}{{\mathfrak H}}
 \nc{\frakI}{{\mathfrak I}}
 \nc{\frakJ}{{\mathfrak J}}
 \nc{\frakK}{{\mathfrak K}}
 \nc{\frakL}{{\mathfrak L}}
 \nc{\frakM}{{\mathfrak M}}
 \nc{\frakN}{{\mathfrak N}}
 \nc{\frakO}{{\mathfrak O}}
 \nc{\frakP}{{\mathfrak P}}
 \nc{\frakQ}{{\mathfrak Q}}
 \nc{\frakR}{{\mathfrak R}}
 \nc{\frakS}{{\mathfrak S}}
 \nc{\frakT}{{\mathfrak T}}
 \nc{\frakU}{{\mathfrak U}}
 \nc{\frakV}{{\mathfrak V}}
 \nc{\frakW}{{\mathfrak W}}
 \nc{\frakX}{{\mathfrak X}}
 \nc{\frakY}{{\mathfrak Y}}
 \nc{\frakZ}{{\mathfrak Z}}
 \nc{\fraka}{{\mathfrak a}}
 \nc{\frakb}{{\mathfrak b}}
 \nc{\frakc}{{\mathfrak c}}
 \nc{\frakd}{{\mathfrak d}}
 \nc{\frake}{{\mathfrak e}}
 \nc{\frakf}{{\mathfrak f}}
 \nc{\frakg}{{\mathfrak g}}
 \nc{\frakh}{{\mathfrak h}}
 \nc{\fraki}{{\mathfrak i}}
 \nc{\frakj}{{\mathfrak j}}
 \nc{\frakk}{{\mathfrak k}}
 \nc{\frakl}{{\mathfrak l}}
 \nc{\frakm}{{\mathfrak m}}
 \nc{\frakn}{{\mathfrak n}}
 \nc{\frako}{{\mathfrak o}}
 \nc{\frakp}{{\mathfrak p}}
 \nc{\frakq}{{\mathfrak q}}
 \nc{\frakr}{{\mathfrak r}}
 \nc{\fraks}{{\mathfrak s}}
 \nc{\frakt}{{\mathfrak t}}
 \nc{\fraku}{{\mathfrak u}}
 \nc{\frakv}{{\mathfrak v}}
 \nc{\frakw}{{\mathfrak w}}
 \nc{\frakx}{{\mathfrak x}}
 \nc{\fraky}{{\mathfrak y}}
 \nc{\frakz}{{\mathfrak z}}
 \nc{\so}{{\mathfrak so}}
 \nc{\sa}{{\mbox{{\scriptsize \cyr x}}}}
 \nc{\slfour}{{\mathfrak sl}_4}
 \nc{\one}{{\bf 1}}
 \nc{\zero}{{\bf 0}}
 \nc{\Qxy}{\Q\langle x,y\rangle}
 \nc{\iij}{{i}}
\nc{\bga}{{\boldsymbol \ga}}
\nc{\gL}{{\Lambda}}

 \nc{\DZ}{\mathcal{DZ}}

\newcommand{\dif}{ \operatorname{\mathfrak D}}

\newcommand{\gr}{ \operatorname{gr}}
\newcommand{\grw}{ \operatorname{gr}^{\operatorname{W};N}}
\newcommand{\grl}{ \operatorname{gr}^{\operatorname{L};N}}
\newcommand{\grwl}{ \operatorname{gr}^{\operatorname{W},\operatorname{L};N}}

\newcommand{\filw}{ \operatorname{Fil}^{\operatorname{W};N}}
\newcommand{\fille}{ \operatorname{Fil}^{\operatorname{L};N} }
\newcommand{\filwle}{ \operatorname{Fil}^{\operatorname{W},\operatorname{L};N}}
\newcommand{\thmref}[1] {Theorem \ref{#1}}
\newcommand{\lemref}[1] {Lemma \ref{#1}}
\newcommand{\corref}[1] {Corollary \ref{#1}}
\newcommand{\propref}[1] {Proposition \ref{#1}}

\newcommand{\qe}[1]{\frac{q^{#1}}{1-q^{#1}}}

\begin{document}
\title[Multiple divisor functions at level 2]{Multiple divisor functions and multiple zeta values at level $N$}

\author{Haiping Yuan}
\address{Department of Physical Sciences, York College of Pennsylvania, York, PA 17403}
\email{hyuan@ycp.edu}

\author{Jianqiang Zhao}
\address{Department of Mathematics,
Eckerd College, St. Petersburg, FL 33711, USA}
\email{zhaoj@ekcerd.edu}

\date{}

\begin{abstract}
Multiple zeta values (MZVs) are generalizations of Riemann zeta values at positive integers to 
multiple variable setting. These values can be further generalized to level $N$ multiple polylog 
values by evaluating multiple polylogs at $N$-th roots of unity. In this paper, we consider 
another level $N$ generalization by
restricting the indices in the iterated sums defining MZVs to congruences classes modulo $N$, which we call the  MZVs at level $N$. The goals of this paper are two-fold. First, we shall lay down the theoretical foundations of these values such as their regularizations and double shuffle relations. Second, we will generalize the multiple divisor functions (MDFs) defined by Bachman and K\"uhn to arbitrary level $N$ and study their relations to MZVs at level $N$. These functions are all $q$-series and similar to MZVs, they have both weight and depth filtrations. But unlike that of MZVs, the product of MDFs usually has mixed weights; however, after projecting to the highest weight we can obtain an algebra homomorphism from MDFs to MZVs. Moreover, the image of the derivation $\dif=q\frac{d}{dq}$ on MDFs vanishes on the MZV side, which gives rise to many nontrivial $\Q$-linear relations. In a sequel to this paper, we plan to investigate the nature of these relations.

\end{abstract}

\maketitle

\section{Introduction}
In the 1730s, long before Riemann, Euler started to study the infinite sums of the following form
$$\zeta(s)=1+\frac1{2^s}+\frac1{3^s}+\frac1{4^s}+\cdots$$
though he only considered $s$ as a real number. Later, in a series
of correspondences with Goldbach in 1740s he also investigated the infinite double sums
$$1+\frac1{2^s}\Big(1+\frac1{2^t}\Big)
+\frac1{3^s}\Big(1+\frac1{2^t}+\frac1{3^t}\Big)+\cdots.$$
In modern notation this is $\zeta(s,t)+\zeta(s+t)$
where $\zeta(m,n)$ is called a double zeta value. In the 1990s, Zagier \cite{Zagier1994}
and Hoffman \cite{Hoffman1992,Hoffman1997} independently extended these objects further
to the multiple zeta values (MZVs for short) which are defined as follows.
For $\bfs=(s_1,\dots,s_d)\in\N^d$
\begin{equation}\label{equ:MZFdefn}
\zeta(\bfs):=\sum_{ k_1>\dots>k_d>0} \frac{1}{k_1^{s_1}\cdots k_d^{s_d}}.
\end{equation}
Clearly this series converges if and only if $s_1\ge 2$. The number $d$ is
called the \emph{depth} (or \emph{length}), denoted by $\dep(\bfs)$, and
$s_1+\dots+s_d$ the \emph{weight}, denoted by $|\bfs|.$

MZVs have been found to play important roles in many areas of mathematics as well as
in physics such as
in the computation of certain Feynman integrals (see \cite{Broadhurst1999}).
Moreover, it is noticed that in fact one needs some generalizations of these numbers to
multiple polylog values at roots of unity (see for e.g., \cite{Broadhurst1996,Broadhurst1999}).
Recall that for any positive integers $n_1,\dots,n_d$,
the multiple polylog in complex variables $x_1,\dots,x_d$ is defined as follows:
\begin{equation}\label{equ:polylog}
Li_{n_1,\dots,n_d}(x_1,\dots,x_d)=\sum_{k_1>\dots>k_d>0}
\frac{x_1^{k_1}\cdots x_d^{k_d}}{ k_1^{n_1}\cdots k_d^{n_d}},
\end{equation}
where $|x_1\cdots x_j|<1 \text{ for }j=1,\dots,d.$
It can be analytically continued to a multi-valued meromorphic function
on $\C^d$ (see \cite{Zhao2007d}). Now fix an $N$-th root of unity
$\eta=\eta_N:=\exp(2\pi \sqrt{-1}/N)$. A multiple polylog value (MPV for short)
at \emph{level} $N$ has the form
\begin{equation}\label{equ:MPVdefn}
\calL_{N}(s_1,\dots,s_d;\ga_1,\dots,\ga_d):=
\Li_{s_1,\dots,s_d}(\eta^{\ga_1},\dots,\eta^{\ga_d})
\end{equation}
for some $s_1,\dots,s_d\in\N$ and $\ga_1,\dots,\ga_d\in\Z/N\Z$.

Extending the works of \cite{GKZ2006,KanekoTa2013,YuanZh2014}, we can provide
another generalization of MZV to level $N$ by restricting the summation indices
in \eqref{equ:MZFdefn} to certain fixed congruence classes modulo $N$.
Suppose $\bfs=(s_1,\dots,s_d)\in\N^d$ and $\bfga=(\ga_1,\dots,\ga_d)\in(\Z/N\Z)^d$.
The  MZVs at level $N$ with \emph{colors} $\bfga$ is
\begin{equation}\label{equ:MZVlevelN}
\zeta_N(\bfs;\bfga) = \sum_{\substack{k_1>\dots>k_d>0\\ k_j\equiv \ga_j\ppmod{N}\ \forall 1\le j\le d} }  \frac{1}{k_1^{s_1}\cdots k_d^{s_d}}.
\end{equation}
In section \ref{sec:relMPVMZV} we will investigate the relation between these values and MPVs defined by \eqref{equ:MPVdefn}. Then we will develop the general theory of regularizations of these values in section~\ref{sec:regularizationMZV}. This will enable us to study their underlying algebra structure.

In \cite{BachmannKu2013} Bachmann and K\"uhn considered the multiple divisor functions (MDFs)
and applied their results to the study of MZVs and multiple Eisenstein series. In this paper,
we generalize their theory to arbitrary levels and study the level $N=2$ case in some detail.
Note that the case $N=d=2$ were considered by Kaneko and Tasaka \cite{KanekoTa2013} and
Nakayama and Tasaka \cite{NakamuraTa2013}. We will see that the correct objects corresponding
to MZVs at arbitrary level $N$ are exactly the MZVs at level $N$.

As for MZVs (see, for e.g., \cite{Hoffman1997})
an extremely effective way to study MDFs is to consider their underlying algebra structure.
In section \ref{sec:MDFs}, by generalizing Bachmann and K\"uhn's work \cite{BachmannKu2013}
at level 1 we will
define the bifiltered algebra $\MD_N$ at arbitrary level $N$ by formalizing the stuffle
product of MDFs. Here, the main complication
results from the fact that these stuffle products often produce MDFs of
lower weights whose coefficients, unlike $N=1$ case, are generally complex numbers lying in
the cyclotomic field $\Q(\eta_N)$. This forces us to study $\MD_N$ as a bifiltered $\Q(\eta_N)$-algebra.
However, all is not lost since when we restrict to the weight-graded pieces we can
safely descend to a $\Q$-algebra (see Thereom~\ref{thm:Zk}).

In some sense, the MDFs are easier to handle than the MZVs at level $N$
because there is a derivation $\dif$ on $\MD_N$ which, when passing to the MZVs' side provides
many double shuffle relations which are supposed to be difficult to find since one would need
two ways to regularize the MZVs at level $N$ first. We will define and compute
$\dif$ in section~\ref{sec:derivation}. We know that for MZVs at higher levels,
double shuffle relations are not sufficient to provide all their $\Q$-linear relations.
In fact, we proved in \cite{Zhao2010b} the distribution relations and the weight one relations
are needed to produce the so-called standard relations and, furthermore, non-standard relations do exist.
To us, the most interesting question is whether one can produce other relations besides
the double shuffle type using the derivation $\dif$. If not, are there any other relations among
MDFs corresponding to the other standard and non-standard relations among MZVs at level $N$.

\medskip
\noindent{\bf Acknowledgement.} This work was started
when both authors were visiting the Morningside Center of Mathematics in Beijing in 2013.
JZ also would like to thank the Max-Planck Institute for Mathematics
and the Kavli Institute for Theoretical Physics China for their hospitality where part of
this work was done. HY is partially supported by the 2013 summer research grant from
York College of Pennsylvania. JZ is partially supported by NSF grant DMS1162116.

\section{Relations between MPVs and MZVs at level  $N$}\label{sec:relMPVMZV}
By definition \eqref{equ:MPVdefn} when $N=1$ MPVs become MZVs.
One of the central themes in the study of MPVs is to determine the $\Q$-linear
relations between them. By using motivic mechanism and higher algebraic $K$-theory
Deligne and Gonccharov \cite{DeligneGo2005} obtained the following result.
\begin{thm}
Let $\MPV_\Q(w,N)$ be the $\Q$-vector space spanned by all the MPVs of weight $w$ and level $N$.
Then we have $d(w,N):=\dim_\Q \MPV_\Q(w,N)\le DG(w,N)$
where $1+\sum_{w=1}^\infty DG(w,N)t^w$ is the
formal power series
\begin{equation*}
\left\{
    \begin{array}{ll}
     \displaystyle{ \frac1{1-t^2-t^3},} & \hbox{if $N=1$;} \\
      \displaystyle{ \frac1{1-t-t^2}, } & \hbox{if $N=2$;} \\
      \displaystyle{ \frac1{1-a(N)t+ b(N)t^2}, } & \hbox{if $N\ge 3$,}
    \end{array}
  \right.
\end{equation*}
where $a(N)=\frac{\varphi(N)}2+\nu(N)$, $b(N)=\nu(N)-1$, $\varphi$
is the Euler's totient function and $\nu(N)$ is the number of distinct
prime divisors of $N$.
\end{thm}
For MZVs (i.e. $N=1$) it is conjectured by Zagier \cite{Zagier1994} that the upper bound
$DG(w,1)$ is sharp and it is widely believed that all the
$\Q$-linear relations are consequences of the so-called double shuffle relations.
At higher levels, however, one needs some extra relations (see \cite{Racinet2002,Zhao2010b}).
In particular, at level $N=2$ some comprehensive computation has been carried out
in \cite{BBV2010}, and at level $N=4$ one needs certain non-standard relations (see \cite{Zhao2008d}).

Another direction of generalizations of MZVs was presented in \cite{YuanZh2014} which was motivated
by \cite{GKZ2006,KanekoTa2013}.
There we defined the \emph{double} zeta values at level $N$ by restricting the summation indices
to congruence classes modulo $N$. However, it is straight-forward to define MZVs at level $N$
of \emph{multiple} integer variables as in \eqref{equ:MZVlevelN}. Here, again, one of the key
problems is to determine all the $\Q$-linear relations among MZVs of fixed weight and level.
So we denote by $\MZV_\Q(w,N)$ the $\Q$-vector space spanned by all the MZVs of weight
$w$ and level $N$. It is not hard too see that when $w=1$ all MZVs at level $N$ diverge.
To remedy the problem we will consider their regularized values in section~\ref{sec:regularizationMZV}.
Notice that Racinet described two ways of regularization of MPVs in \cite{Racinet2002},
one by using series representation (we say such regularized values $\calL^*$ are $*$-regularized)
and the other by integral representation
(we say such regularized values  $\calL^\sha$ are $\sha$-regularized).
Let $\MPV^*_\Q(w,N)$ be the $\Q$-vector space generated by
all $*$-regularized MPVs, and similarly define $\MPV^\sha_\Q(w,N)$,
$\MZV^*_\Q(w,N)$ and $\MZV^\sha_\Q(w,N)$ (see section~\ref{sec:regularizationMZV}).

\begin{conj}\label{conj:MZV=MPVLevelN}
For all weight $w$ and level $N$ we have
$$\dim_Q \MZV^*_\Q(w,N)=\dim_\Q \MZV^\sha_\Q(w,N)=\dim_\Q \MPV^*_\Q(w,N)=\dim_\Q \MPV^\sha_\Q(w,N).$$
\end{conj}

Notice that at level $N=2$ the MPVs become the so called alternating Euler sums.
As usual, instead of writing $\calL_2(s_1,\dots,s_d;a_1,\dots,a_d)$ we will write
$\zeta(s_1,\dots,s_d)$ and then add bars over $s_j$'s for all $a_j=1$. For example,
the alternating series $\calL_2(s,t;1,0)$ is denoted by $\zeta(\bar{s},t)$. The
following beautiful identity involving alternating Euler sums was first proved
in \cite{BorweinBr2006} for $n=1$ and in \cite{Zhao2010a} for arbitrary $n$:
\begin{equation*}
\zeta(\{3\}^n)=8^n\zeta(\{\ol2,1\}^n).
\end{equation*}

\section{Regularizations of MZVs at level $N$}\label{sec:regularizationMZV}
It turns out that it is beneficial to us to have divergent MZVs at level $N$ at our disposal.
Similar to MZVs treated in \cite{IKZ2006}, or more generally, MPVs at level $N$ in \cite{Racinet2002},
one can use two ways to regularize divergent MZVs at level $N$.
When $N=d=2$ this was done in \cite{YuanZh2014}. In this section, we deal with the general case.

\begin{defn}\label{defn:gG}
We define $\gG_0=0$ and for all $1\le \gb<N$
\begin{equation*}
     \gG_\gb=\gG_\gb(\emptyset;\emptyset)=-\frac{1}{\gb}+
     \sum_{n=1}^\infty \left(\frac{1}{nN}-\frac{1}{\gb+nN} \right).
\end{equation*}
Let $d>0$, $\bfs=(s_1,\dots,s_d)\in\N^d$ and $\bfga=(\ga_1,\dots,\ga_d)\in(\Z/N\Z)^d$.
Define $\gG_0(\bfs;\bfga):=0$. Set $\ga_0=0$ and for all $1\le \gb<N$ define
\begin{align}
\gG_\gb(\bfs;\bfga)
:=&\,
\left\{
  \begin{array}{ll}
    {\displaystyle \sum_{\substack{n_0>n_1 > \dots > n_d > 0\\ n_j\equiv \ga_j\ppmod{N}\ \forall j} }    \left(\frac{1}{n_0}-\frac{1}{\gb+n_0} \right)
\frac{1}{n_1^{s_1}\cdots n_d^{s_d}} }, & \hbox{$\displaystyle {{\text{if } \gb<\ga_1} \atop{\text{or } d=0;}}$} \\
    {\displaystyle\sum_{\substack{n_0>n_1 > \dots > n_d > 0\\ n_j\equiv \ga_j\ppmod{N}\ \forall j} }   \left(\frac{1}{n_0}-\frac{1}{\gb-N+n_0} \right)
\frac{1}{n_1^{s_1}\cdots n_d^{s_d}} }, & \hbox{otherwise.}
  \end{array}
\right.
\label{equ:gG_gb}
\end{align}
Straightforward computation shows that all these values are finite.
Let $\MZV_N$ be the $\Q$-vector space generated by all convergent MZVs of level $N$ together
with all the values $\gG_\gb(\bfs;\bfga)$ including $(\bfs;\bfga)=(\emptyset;\emptyset)$.
\end{defn}

\begin{defn}\label{defn:regularize}
We define the $*$-regularized version of $\zeta_N(1;0)$ as a polynomial in $T$ of degree 1:
\begin{equation*}
\zeta^\ast_N(1;0):=\frac1N \left(T+\sum_{\gb=0}^{N-1}\gG_\gb \right)\in \MZV_N[T]
\end{equation*}
and for all $1\le \gb<N$ we define
\begin{equation*}
\zeta^\ast_N(1;\gb):=\zeta^\ast_N(1;0)-\gG_\gb \in \MZV_N[T].
\end{equation*}
If $s_1>1$ then we define the $*$-regularized version of $\zeta_N(1,\bfs;\gb,\bfga)$
as a polynomial in $T$ of degree 1 by using the stuffle relations:
\begin{multline*}
\zeta^\ast_N(1,\bfs;\gb,\bfga):=\zeta^\ast_N(1;\gb)\zeta_N(\bfs;\bfga)
-\sum_{j=1}^d \gd_{\ga_j,\gb} \zeta_N(s_1,\dots,s_j+1,s_{j+1},\dots,s_d;\bfga)\\
-\sum_{j=1}^d \zeta_N(s_1,\dots,s_j,1,s_{j+1},\dots,s_d;\ga_1,\dots,\ga_j,\gb,\ga_{j+1},\dots,\ga_d)\in \MZV_N[T].
\end{multline*}
Now suppose $m\ge 1$, $d\ge 0$ and suppose $\zeta^\ast_N(\{1\}^m,\bfs;\bfgb,\bfga)$ have been defined
as a polynomial in $\MZV_N[T]$ of degree at most $m$
(in fact exactly $m$, but we leave this to the interested reader to check)
for all $\bfgb'=(\gb_1,\dots,\gb_m)\in(\Z/N\Z)^m$,
$\bfga\in(\Z/N\Z)^d$ and $\bfs=(s_1,\dots,s_d)\in\N^d$ with $s_1>1$.
In the following we use the stuffle relations to define $\zeta^\ast_N(\{1\}^{m+1},\bfs;\bfgb,\bfga)$
for all $\bfga=(\gb_0,\dots,\gb_m)\in(\Z/N\Z)^{m+1}$ by some linear algebra.

First we write down the equation
\begin{equation*}
    T\cdot \zeta^\ast_N(\{1\}^m,\bfs;\bfgb',\bfga)
     =\sum_{\gam=0}^{N-1} \zeta^\ast_N(1;\gam)\zeta^\ast_N(\{1\}^m,\bfs;\bfgb',\bfga).
\end{equation*}
Expanding the right hand side using stuffle relations we see that
\begin{equation*}
\sum_{\ell=0}^{m}\sum_{\gam=0}^{N-1} \zeta^\ast_N(\{1\}^{m+1},\bfs;\gb_1,\dots,\gb_\ell,\gam,\gb_{\ell+1},\dots,\gb_m,\bfga)
\end{equation*}
is in $\MZV_N[T]$  of degree at most $m+1$.
On the other hand we write down the equations
\begin{equation*}
\sum_{\gam=0}^{N-1} \gG_\gam(\{1\}^m,\bfs;\bfgb',\bfga) =N\zeta^\ast_N(\{1\}^{m+1},\bfs;0,\bfgb',\bfga)
-\sum_{\gam=0}^{N-1} \zeta^\ast_N(\{1\}^{m+1},\bfs;\gam,\bfgb',\bfga).
\end{equation*}
Adding up the above two equations we get the following element in $\MD_N[T]$
\begin{equation*}
N\zeta^\ast_N(\{1\}^{m+1},\bfs;0,\bfgb',\bfga)
+\sum_{\ell=1}^{m}\sum_{\gam=0}^{N-1} \zeta^\ast_N(\{1\}^{m+1},\bfs;
    \gb_1,\dots,\gb_\ell,\gam,\gb_{\ell+1},\dots,\gb_m,\bfga).
\end{equation*}
Subtracting $N \gG_{\gb_0}(\{1\}^m,\bfs;\bfgb',\bfga)$ we see that for all
$\bfgb=(\gb_0,\dots,\gb_m)\in(\Z/N\Z)^{m+1}$
\begin{equation*}
N\zeta^\ast_N(\{1\}^{m+1},\bfs;\bfgb,\bfga)
+\sum_{\ell=1}^{m}\sum_{\gam=0}^{N-1} \zeta^\ast_N(\{1\}^{m+1},\bfs;\gb_1,\dots,\gb_\ell,\gam,\gb_{\ell+1},\dots,\gb_m,\bfga)
=b(\bfgb),
\end{equation*}
for some $b(\bfgb)\in \MD_N[T]$ of degree at most $m+1$.
Regarding $\zeta^\ast_N(\{1\}^{m+1},\bfs;\bfgb,\bfga)$ as $N^{m+1}$ variables as $\bfgb$ varies
we now only need to show the corresponding $N^{m+1}\times N^{m+1}$ coefficient
matrix $M(N,m)$ is nonsingular. Moreover, the inverse of $M(N,m)$ has entries in $\Q$.
Thus $M(N,m)^{-1} [b(\bfgb)]_{\bfgb\in \in(\Z/N\Z)^{m+1}}$ provides the definition
of $\zeta^\ast_N(\{1\}^{m+1},\bfs;\bfgb,\bfga)$ for all $\bfgb\in(\Z/N\Z)^{m+1}$ at the same time.

Using the $N$-adic system to order the variables
by assigning the position number $\sum_{r=0}^{m-1} \gb_r N^r$ to
$\zeta^\ast_N(\{1\}^{m+1},\bfs;\gb_m,\dots,\gb_0,\bfga)$ we see that
\begin{equation*}
M(N,m)=N I_{N^{m+1}}+\sum_{r=0}^{m-1} \Big(E^{(m,r)}_{i,j} \Big)_{0\le i,j< N^{m+1} },
\end{equation*}
where $I_{N^{m+1}}$ is the identity matrix of size $N^{m+1}$ and
$E^{(m,r)}$ is an $N^{m+1}\times N^{m+1}$ matrix whose $(i,j)$-th entry is given by
\begin{equation*}
E^{(m,r)}_{i,j}=
\left\{
  \begin{array}{ll}
    1, & \hbox{if $N^r|(j-\bar{\imath} N)$ and  $0\le j-\bar{\imath} N<N^r(N-1)$;} \\
    0, & \hbox{otherwise,}
  \end{array}
\right.
\end{equation*}
where $0\le \bar{\imath}<N^m$ such that $\bar{\imath}\equiv i \pmod{N^m}$.
We now prove that
\begin{equation}\label{equ:MatrixMNm}
 \det M(N,m)=N^{N^{m+1}} \cdot (m+1)\cdot \prod_{j=2}^m  j^{N^{m-j}(N-1)}.
\end{equation}
Partition the matrix $M(N,m)$ into $N$ rows and $N$ columns of smaller square blocks of size $N^m\times N^m$.
Subtracting the first (0th, strictly speaking) row of these blocks from all the other rows we
get the matrix of the form
\begin{equation*}
\begin{pmatrix}
* & * & * & \cdots & *  \\
-N I_{N^m} & N I_{N^m} & 0 & \cdots & 0    \\
-N I_{N^m} & 0 & N I_{N^m} & \cdots & 0   \\
\vdots &\vdots &\ddots &\vdots &\vdots \\
-N I_{N^m} & 0 & 0 & \cdots & N I_{N^m}
\end{pmatrix}.
\end{equation*}
Now adding all the columns of these blocks to the first we get
\begin{equation*}
\begin{pmatrix}
M_1 & * & * & \cdots & *   \\
0 & N I_{N^m} & 0 & \cdots & 0    \\
0 & 0 & N I_{N^m} & \cdots & 0   \\
\vdots &\vdots &\ddots &\vdots &\vdots \\
0 & 0 & 0 & \cdots & N I_{N^m}
\end{pmatrix},
\end{equation*}
where $M_1$ is $N^{m}\times N^{m}$ matrix with the form
$$M_1= N I_{N^m}+\sum_{r=0}^{m-1} \Big(E^{(m-1,r)}_{i,j} \Big)_{0\le i,j< N^m }.$$
So  $\det M(N,m)=N^{N^m(N-1)}\det M_1.$
Applying the same trick on $M_1$ we can transform it to the $N^m\times N^m$ matrix of the form
\begin{equation*}
\begin{pmatrix}
M_2 & * & * & \cdots & *   \\
0 & N I_{N^{m-1}} & 0 & \cdots & 0    \\
0 & 0 & N I_{N^{m-1}} & \cdots & 0   \\
\vdots &\vdots &\ddots &\vdots &\vdots \\
0 & 0 & 0 & \cdots & N I_{N^{m-1}}
\end{pmatrix}.
\end{equation*}
So  $\det M(N,m)=N^{N^{m+1}-N^{m-1}}  \det M_2.$ Repeatedly using this idea and we get
for $j=2,\dots, m$
\begin{equation*}
 \det M(N,m)=N^{N^{m+1}-N^{m-1}} (2N)^{N^{m-2}(N-1)}  \cdot \big((j-1)N\big)^{N^{m-j+1}(N-1)} \det M_j
\end{equation*}
where $M_{j}$ is $N^{m-j+1}\times N^{m-j+1}$ matrix with the form
$$M_{j}= (jN) I_{N^{m-j+1}}+\sum_{r=0}^{m-j} \Big(E^{(m-j,r)}_{i,j} \Big)_{0\le i,j< N^{m-j+1} }$$
This can be proved by induction on $j$. Hence when $j=m$ we get
\begin{equation*}
 \det M(N,m)=N^{N^{m+1}-N^{m-1}} (2N)^{N^{m-2}(N-1)}  \cdot \big((m-1)N\big)^{N(N-1)} \det M_m
\end{equation*}
where $M_m$ is the $N\times N$ matrix with the following form
\begin{equation*}
\begin{pmatrix}
mN+1 & 1 & 1 & \cdots & 1   \\
1 & mN+1  & 1 & \cdots & 1    \\
1 & 1 & mN+1 & \cdots & 1   \\
\vdots &\vdots &\ddots &\vdots &\vdots \\
1 & 1 & 1 & \cdots & mN+1
\end{pmatrix}.
\end{equation*}
Adding all the rows to the first row we get $\det M_m=(m+1)N  \det M_m'$ where
\begin{equation*}
M_m'=\begin{pmatrix}
1 & 1 & 1 & \cdots & 1   \\
1 &  mN+1 & 1 &\cdots & 1   \\
1 & 1 & mN+1 & \cdots & 1   \\
\vdots &\vdots &\ddots &\vdots &\vdots \\
1 & 1 & 1 & \cdots & mN+1
\end{pmatrix}.
\end{equation*}
Subtracting the first row from all the other rows we get $\det M_m'=\det M_m''$ where
\begin{equation*}
M_m''=\begin{pmatrix}
1 & 1 & 1 & \cdots & 1   \\
0 & mN & 0 & \cdots & 0    \\
0 & 0 & mN & \cdots & 0   \\
\vdots &\vdots &\ddots &\vdots &\vdots \\
0 & 0 & 0 & \cdots & mN
\end{pmatrix}.
\end{equation*}
Thus
\begin{equation*}
 \det M(N,m)=N^{N^{m+1}-N^{m-1}} (2N)^{N^{m-2}(N-1)}
 \big((m-1)N\big)^{N(N-1)} (mN)^{N-1} (m+1) N.
\end{equation*}
Simplifying we arrive at \eqref{equ:MatrixMNm}.
\end{defn}

\begin{eg} \label{eg:N=2m=2}
When $N=2$ and $m=1$ we have $\gG_1=-\ln(2)$ and
\begin{equation*}
\zeta^\ast_2(1;0):= \frac12 (T-\ln(2)), \quad \zeta^\ast_2(1;1):= \frac12 (T+\ln(2)).
\end{equation*}
When $m=2$ we first write down
\begin{align}
T\cdot \zeta^\ast_2(1;0)=&\, 2\zeta^\ast_2(1,1;0,0)+\zeta^\ast_2(1,1;0,1)+\zeta^\ast_2(1,1;1,0)+\zeta_2(2;0),\label{equ:10A}\tag{$A_{1,0}$}\\
T\cdot \zeta^\ast_2(1;1)=&\, \zeta^\ast_2(1,1;0,1)+\zeta^\ast_2(1,1;1,0)+2\zeta^\ast_2(1,1;1,1)+\zeta^\ast_2(2;1).\label{equ:11A}\tag{$A_{1,1}$}
\end{align}
Then we write down
\begin{align}
\gG_1(1;0) =&\, \zeta^\ast_2(1,1;0,0)-\zeta^\ast_2(1,1;1,0),\label{equ:10B}\tag{$B_{1,0}$}\\
\gG_1(1;1) =&\, \zeta^\ast_2(1,1;0,1)-\zeta^\ast_2(1,1;1,1).\label{equ:11B}\tag{$B_{1,1}$}
\end{align}
Computing \eqref{equ:10A}$\pm$\eqref{equ:10B} and \eqref{equ:11A}$\pm$\eqref{equ:11B} we get
\begin{align*}
3\zeta^\ast_2(1,1;0,0)+\zeta^\ast_2(1,1;0,1)=&\, T\cdot \zeta^\ast_2(1;0)+\gG_1(1;0) -\zeta_2(2;0),\\
2\zeta^\ast_2(1,1;0,1)+ \zeta^\ast_2(1,1;1,0)+ \zeta^\ast_2(1,1;1,1)=&\, T\cdot \zeta^\ast_2(1;1)+\gG_1(1;1) -\zeta_2(2;1),\\
\zeta^\ast_2(1,1;0,0)+2\zeta^\ast_2(1,1;1,0)+\zeta^\ast_2(1,1;0,1)=&\, T\cdot \zeta^\ast_2(1;0)-\gG_1(1;0) -\zeta_2(2;0),\\
3\zeta^\ast_2(1,1;1,1)+ \zeta^\ast_2(1,1;1,0)=&\, T\cdot \zeta^\ast_2(1;1)-\gG_1(1;1) -\zeta_2(2;1).
\end{align*}
Thus
\begin{equation*}
\begin{bmatrix}
  3 & 1 & 0 & 0 \\
  0 & 2 & 1 & 1 \\
  1 & 1 & 2 & 0 \\
  0 & 0 & 1 & 3
\end{bmatrix}
\begin{bmatrix}
  \zeta^\ast_2(1,1;0,0) \\
  \zeta^\ast_2(1,1;0,1) \\
  \zeta^\ast_2(1,1;1,0) \\
  \zeta^\ast_2(1,1;1,1)
\end{bmatrix}
=\begin{bmatrix}
   T\cdot \zeta^\ast_2(1;0)+\gG_1(1;0) -\zeta_2(2;0) \\
   T\cdot \zeta^\ast_2(1;1)+\gG_1(1;1) -\zeta_2(2;1) \\
   T\cdot \zeta^\ast_2(1;0)-\gG_1(1;0) -\zeta_2(2;0) \\
   T\cdot \zeta^\ast_2(1;1)-\gG_1(1;1) -\zeta_2(2;1)
 \end{bmatrix}.
\end{equation*}
Hence
\begin{align*}
\begin{bmatrix}
  \zeta^\ast_2(1,1;0,0) \\
  \zeta^\ast_2(1,1;0,1) \\
  \zeta^\ast_2(1,1;1,0) \\
  \zeta^\ast_2(1,1;1,1)
\end{bmatrix}
=&\,
\frac1{16}
\begin{bmatrix}
  \phantom{-} 5 & -3 & \phantom{-} 1 & \phantom{-} 1 \\
  \phantom{-} 1 & \phantom{-} 9 & -3 & -3 \\
  -3 & -3 & \phantom{-} 9 & \phantom{-} 1 \\
  \phantom{-} 1 & \phantom{-} 1 & -3 & \phantom{-} 5
\end{bmatrix}
\begin{bmatrix}
   \frac{T}2(T+\gG_1)+\gG_1(1;0) -\zeta_2(2;0) \\
   \frac{T}2(T-\gG_1)+\gG_1(1;1) -\zeta_2(2;1) \\
   \frac{T}2(T+\gG_1)-\gG_1(1;0) -\zeta_2(2;0) \\
   \frac{T}2(T-\gG_1)-\gG_1(1;1) -\zeta_2(2;1)
 \end{bmatrix}    \\
 =&\, \frac1{16}
\begin{bmatrix}
2 T^2+4 \gG_1 T+4 \gG_1(1;0)-4 \gG_1(1;1)-6 \zeta_2(2;0)+2 \zeta_2(2;1)\\
2 T^2-4 \gG_1 T+4 \gG_1(1;0)+12 \gG_1(1;1)+2 \zeta_2(2;0)-6 \zeta_2(2;1)\\
2 T^2+4 \gG_1 T-12 \gG_1(1;0)-4 \gG_1(1;1)-6 \zeta_2(2;0)+2 \zeta_2(2;1)\\
2 T^2-4 \gG_1 T+4 \gG_1(1;0)-4 \gG_1(1;1)+2 \zeta_2(2;0)-6 \zeta_2(2;1)
\end{bmatrix}.
\end{align*}
By comparing with the stuffle relation
\begin{equation*}
\frac14(T+\gG_1)^2=\zeta^\ast_2(1;0)^2=2\zeta^\ast_2(1,1;0,0)+\zeta_2(2;0)
\end{equation*}
we get a relation
\begin{equation*}
\gG_1^2+2\gG_1(1;1)-2\gG_1(1;0)=\zeta_2(2;1)+\zeta_2(2;0)=\zeta(2),
\end{equation*}
which can also be proved by the fact that $\gG_1^2=\zeta(\bar1)^2=2\zeta(\bar1,\bar1)+\zeta(2)$,
$2\gG_1(1;0)=\zeta(\bar1,1)+\zeta(\bar1,\bar1)$ and $2\gG_1(1;1)=\zeta(\bar1,1)-\zeta(\bar1,\bar1)$.
\end{eg}

We leave the interested reader to check that the above Definition~\ref{defn:regularize} is
consistent with the $*$-regularized values of MPVs
given by Racinet in \cite{Racinet2002}. Namely, \eqref{equ:MZV=MPVLevelN} can
be extended to these regularized values:
\begin{equation}\label{equ:MZV=MPVLevelNReg}
\zeta^\ast_N(\bfs;\bfga)=\frac1{N^d} \sum_{a_1=1}^N\cdots \sum_{a_d=1}^N
    \eta^{-(a_1\ga_1+\cdots+a_d\ga_d)} \calL^\ast_N(\bfs; a_1,\dots,a_d)
\end{equation}
for all $\bfga\in\Z/N\Z$ and $\bfs\in\N^d$ with $s_1=1$.

Notice that Conjecture~\ref{conj:MZV=MPVLevelN} becomes easy to verify if we extend the scalars to
the cyclotomic field $\Q(\eta_N)$. Conjecturally, every linear relations between MPVs
over $\overline{\Q}$ (the algebraic closure of $\Q$) must be a consequence of some linear
relations over $\Q$ (the same should be true for MZVs at level $N$). This is the motivation
for our conjecture in the first place. To show the $\Q(\eta_N)$-version of the conjecture,
using the identity
\begin{equation*}
    \sum_{j=0}^{N-1}  \eta^{Mj}= \left\{
                              \begin{array}{ll}
                                N, & \hbox{if $N|M$;} \\
                                \, 0, & \hbox{otherwise,}
                              \end{array}
                            \right.
\end{equation*}
we can easily find
\begin{equation}\label{equ:MZV=MPVLevelN}
\begin{split}
\zeta_N(\bfs;\bfga) &\, =\frac{1}{N^d}   \sum_{k_1>\dots>k_d>0}
 \sum_{a_1=1}^N \eta^{a_1(k_1-\ga_1)} \cdots \sum_{a_d=1}^N \eta^{a_d(k_d-\ga_d)}
\frac{1}{k_1^{s_1}\cdots k_d^{s_d}} \\
&\, =\frac1{N^d} \sum_{a_1=1}^N\cdots \sum_{a_d=1}^N  \eta^{-(a_1\ga_1+\cdots+a_d\ga_d)}
\calL_N(\bfs; a_1,\dots,a_d).
\end{split}
\end{equation}
Let $h=\eta^{-1}$. It is not difficult to see that the
determinant of the $N^d\times N^d$ matrix
$$\Big[h^{a_1\ga_1+\cdots+a_d\ga_d}\Big]_{1\le a_1,\dots,a_d,\ga_1,\dots,\ga_d\le N}$$
is given by $\det(V)^{N^{2d-2}}$ where $V$ is the Vandermonde matrix
$\big[h^{a \ga} \big]_{1\le a,\ga\le N}$. Set $f(x)=x^N-1$. Then
$$(-1)^{\binom{N}{2}}\det(V)^2=\prod_{1\le a\ne b\le N} (h^a-h^b)=\prod_{a=1}^N f'(h^a)
= N^N  \left(\prod_{a=1}^N h^a\right)^{N-1} =(-1)^{N-1} N^N.
$$
Therefore $\det(V)^2=\pm N^N$ where we take $+$ for $N\equiv 2\pmod{4}$ and $-$ otherwise.
Consequently, for every fixed
$(a_1,\dots,a_d)\in(\Z/N\Z)^d$ the MPV $\calL_N(\bfs; a_1,\dots,a_d)$ can be expressed
as a $\Q(\eta_N)$-linear combinations of $\zeta_N(\bfs;\bfga)$ ($\bfga\in(\Z/N\Z)^d$).
The above argument proves the following important fact.
\begin{thm}\label{thm:MZVMPVlevelN}
Let $\Q_N=\Q(\eta_N)$ be the $N$-th cyclotomic field. As $\Q_N$-vector spaces we have
$$\dim_{\Q_N} \MZV^*_\Q(w,N)\otimes \Q_N =\dim_{\Q_N} \MPV^*_\Q(w,N)\otimes \Q_N.$$
\end{thm}

The $\sha$-regularized version of $\zeta(1; \alpha)$ is defined as follows.

\begin{defn}
$$\zeta^{\sha}(1; \alpha)= \frac{1}{N}\left(T+\sum_{n=1}^{N-1}  \eta^{-n\ga} Li_1(\eta^n)
\right).$$
\end{defn}

Now let $m\ge 1$, $d\ge 0$ and suppose $\calL^\sha_N(\{1\}^m,\bfs;\bfgb',\bfga)$ have been defined
as a polynomial in $\MZV_N[T]$ of degree at most $m$
for all $\bfgb'=(\gb_1,\dots,\gb_m)\in(\Z/N\Z)^m$,
$\bfga\in(\Z/N\Z)^d$ and $\bfs=(s_1,\dots,s_d)\in\N^d$ with $s_1>1$.  Using the partial fractions $$\frac{t^{a-1}}{1-t^N} =\frac1N \sum_{m=1}^N \frac{\eta^{-(a-1)m}}{1-\eta^{m}t},$$ we obtain
\begin{align*}
\zeta^{\sha}_N(\bfs; \bfga)&= \int_0^1 \omega^{s_1-1}\frac{t^{\alpha_1-\alpha_2-1}dt}{1-t^N} \cdots \omega^{s_d-1} \frac{t^{\alpha_d-1}dt}{1-t^N}\\
&= \frac{1}{N^d} \sum_{m_1, \cdots, m_d=1}^{N}\eta^{-\sum_{i=1}^{d-1}m_i(\alpha_i-\alpha_{i+1}-1)-(\alpha_d-1)m_d}
\calL^{\sha}_N(\bfs; \eta^{\alpha_1}, \cdots, \eta^{\alpha_d}),
\end{align*}
for $\bfs\in \N^n$ and $\bfga\in (\Z/N\Z)^n.$
Thus for $\bfs= (\{1\}^{m+1}, \bfs')$ with $\bfs'=(s_1, \cdots, s_{d})\in \N^{d}$ and $s_1>1$, $\bfga=(\alpha_1, \cdots, \alpha_d)\in (\Z/N\Z)^{d}$ and $\bfgb=(\gb_0,\cdots, \gb_m)\in (\Z/N\Z)^{m+1}$,
\begin{multline*}
\zeta^{\sha}_N(\{1\}^{m+1}, \bfs'; \bfb, \bfga) = \frac{1}{N^d} \sum_{m_1, \cdots, m_d=1}^{N}\eta^{-\sum_{i=1}^{d-1}m_i(\alpha_i-\alpha_{i+1}-1)-(\alpha_d-1)m_d}\\
\cdot\calL^{\sha}_N(\{1\}^{m+1},\bfs'; \eta^{\gb_0},\cdots,\eta^{\gb_m},\eta^{\alpha_1}, \cdots, \eta^{\alpha_d}),
\end{multline*}
where
\begin{multline*}
\calL^{\sha}_N(\{1\}^{m+1},\bfs'; \eta^{\gb_0},\cdots,\eta^{\gb_m},\eta^{\alpha_1}, \cdots, \eta^{\alpha_d})\\
 = \frac{1}{m+1}[ Li_{1}^{\sha}(\eta^{\gb_0})\cdot \calL^{\sha}_{N} (\{1\}^{m}, \bfs'; \eta^{\gb_1},\cdots,\eta^{\gb_m},\eta^{\alpha_1}, \cdots, \eta^{\alpha_d})\\
- \sum_{j=1}^{d-m-1}\sum_{t_j=2}^{s_j} \calL^{\sha}_{N}(\{1\}^{m}, s_1, \cdots, t_j, s_j+1-t_j, \cdots, s_{d-m-1};\eta^{\gb_0},\cdots,\eta^{\gb_m}, \eta^{\alpha_1}, \cdots, \eta^{\alpha_d}) \\
- \calL^{\sha}_{N}(\{1\}^m, \bfs', 1; \eta^{\gb_0},\cdots,\eta^{\gb_m},\eta^{\alpha_1}, \cdots, \eta^{\alpha_d})]
\end{multline*}
with $Li^{\sha}_1(\eta^{\alpha_1}) = \left\{
                                          \begin{array}{ll}
                                            T, & \hbox{if $\alpha_1 \equiv 0 \pmod{N}$;} \\
                                            Li_1(\eta^{\alpha_1}), & \hbox{if $\alpha_1 \not\equiv 0\pmod{N}$.}
                                          \end{array}
                                        \right.$\\

Next we want to study the relation between the above two different regulations. Our goal is to show that both regularizations still satisfy the same relation obtained by Ihara, Kaneko,and Zagier 
(see \cite[Theorem 1]{IKZ2006}). By \eqref{equ:MZV=MPVLevelNReg}, it suffices to prove the same relation for MPV $\calL_N(\boldsymbol{s}; a_1, \cdots, a_d),$ which has been done by Racinet \cite{Racinet2002}.

We now recall some of the results from \cite{Racinet2002}. For convenience we use Racinet's notation. 
For a set $Z$ and a field $k$, let $k\langle Z\rangle$ (resp.\ $\mathfrak{Lib}_k(Z)$) be the free (noncommutative) associative algebra (resp.\ the free Lie algebra) on $Z$ with coefficients in $k$. We denote by $k\langle\!\langle Z\rangle\!\rangle$ the set of all formal series on $Z$ over the field $k$. The algebra $k\langle Z\rangle$, when viewed as the enveloping  bigebra of $\mathfrak{Lib}_k(Z)$, has the coproduct $\Delta$ defined by $\Delta z = 1\otimes z + z\otimes 1,$ for $\forall z\in Z.$

Let $\gG$ be a multiplicative finite group of $\C^{\times}$. In this paper, one may take $\gG$ to be the group of $N$-th roots of unity. Then one defines the alphabet $X:=\{x_\gs: \gs\in\gG\cup \{0\}\}.$
For $(n, \nu)\in \N\times\gG$, let $y_{n, \nu}= x_0^{n-1}x_{\nu}$. We denote by $Y$ the set of all $y_{n, \nu}$ for $(n, \nu)\in \N\times\gG.$ The subalgebra of $k\langle X\rangle$ generated by $Y$ is denoted by $k\langle Y\rangle$ and is free on $Y$. As a vector space, it is generated by the words on $X$ not ending by $x_0$ thus it could be identified with the quotient $k\langle X\rangle/I_0$, where $I_0= k \langle X \rangle x_0.$ The corresponding projection is denoted by $\pi_Y.$  The vector subspace $k \langle X\rangle_{cv}$ of $k \langle X\rangle$ generated by words on $X$ that are not ending by $x_0$ and not starting by $x_1$ is also a graded algebra. It may be identified with the quotient $k\langle X\rangle/I$, where $I= (x_1 k \langle X \rangle + k \langle X\rangle x_0)$. We denote by $\pi_{cv}$ the corresponding projection. The algebra $k\langle X\rangle_{cv}$ is same as the subalgebra of $k\langle Y\rangle$ generated by words on $Y$ that are not starting by $y_{1,1}=x_1.$ Thus it can be identified with the quotient $k \langle Y\rangle/ y_{1,1}k\langle Y\rangle,$ which is denoted by $k\langle Y\rangle_{cv}.$ Again its corresponding projection is denoted by $\pi_{cv}.$ Similarly one defines $k\langle\!\langle X\rangle\!\rangle_{cv}$ and $k\langle\!\langle Y\rangle\!\rangle_{cv}.$

We extend the definition of $y_{n, \nu}$ to the case where $n=0$
$$y_{0, \nu}=
\left\{
  \begin{array}{ll}
    1, & \hbox{if $\nu =1$;} \\
    0, & \hbox{if $\nu\neq 1$.}
  \end{array}
\right.$$
Then define the coproduct $$\Delta_{\ast} (y_{n, \nu})= \sum_{\substack{k+l=n\\ \kappa +\lambda=\nu}}y_{k, \kappa} \otimes y_{l, \lambda}.$$ It's easy to check (see \cite{Racinet2002}) that $(k\langle Y\rangle, \cdot, \Delta_{\ast})$ is a bialgebra. Similarly, one can check that  the associative algebra $k\langle\!\langle X\rangle\!\rangle$ is a bialgebra (see \cite{Reutenauer1993}). So are $k\langle\!\langle Y\rangle\!\rangle, k\langle\!\langle X\rangle\!\rangle_{cv}$ and $k\langle\!\langle Y\rangle\!\rangle_{cv}$. Recall that an element $x$ in any coalgebra $(C, \epsilon, \Delta)$ is called \emph{diagonal} if $\Delta(x)= x\otimes x$ and $\epsilon(x)=1.$

For any word $y_{s_1, \gs_1}y_{s_2, \gs_2}\cdots y_{s_r, \gs_r}$ on $Y$, we define the endomorphism $\bfq$ of $k\langle Y\rangle$ as  $$\bfq(y_{s_1, \gs_1}y_{s_2, \gs_2}\cdots y_{s_r, \gs_r})= y_{s_1, \gs_1}y_{s_2, \gs_2\gs_1^{-1}} \cdots y_{s_r, \gs_r\gs_{r-1}^{-1}}.$$

Recall the multiple polylog  $Li_{s_1, \cdots, s_d}(z_1, \cdots, z_d)$ is defined by
\begin{equation}
Li_{s_1,\dots,s_d}(z_1,\dots,z_d)=\sum_{k_1>\dots>k_d>0}
\frac{z_1^{k_1}\cdots z_d^{k_d}}{ k_1^{s_1}\cdots k_d^{s_d}},
\end{equation}
Let $$I_{[0, 1]}(\gs_1, \cdots, \gs_r)= \int_{1\geq t_1 \geq \cdots t_r\geq 0} \bigwedge_{i=1}^r \omega_{\gs_i}(t_i),$$ where $$\omega_{\gs_i}(t_i)=
\left\{
  \begin{array}{ll}
    \frac{d t_i}{\gs_i^{-1} -t_i}, & \hbox{if $\gs_i \neq 0$;} \\
    \frac{dt_i}{t_i}, & \hbox{if $\gs_i=0$.}
  \end{array}
\right.$$ Then $$Li_{s_1, \cdots, s_r} (1, \cdots, 1) = \zeta(s_1, \cdots, s_r)= I_{[0, 1]}(\{0\}^{s_1-1}, 1, \cdots, \{0\}^{s_{r}-1}, 1).$$

The generating series of $Li_{s_1, \cdots, s_r}(z_1, \cdots, z_r)$ and $I_{[0, 1]}(\gs_1, \cdots, \gs_r)$ are defined as follows:
\begin{align*}
\calL_{cv}:=& \sum_{\substack{s_1, \cdots, s_r\in\N, \gs_1, \cdots, \gs_r\in\gG \\
                              (s_1, \gs_1)\ne (1,1)}}
        Li_{s_1, \cdots, s_r}(\gs_1, \cdots, \gs_r) y_{s_1, \gs_1} \cdots y_{s_r, \gs_r},\\
\calI_{cv}:=& \sum_{\gs_{1}, \cdots, \gs_{r-1}\in\gG\cup\{0\}, \gs_r\in\gG}
        I_{[0, 1]}(\gs_1, \cdots, \gs_r) x_{\gs_1} \cdots x_{\gs_r}.
\end{align*}
It is clear one can think of them as two elements in $\C\langle\!\langle Y\rangle\!\rangle_{cv}$ and
$\C\langle\!\langle X\rangle\!\rangle_{cv}$, respectively. Racinet further shows
\begin{prop}
The series $\calL_{cv}$ and $\calI_{cv}$ are diagonal elements in $\C\langle\!\langle Y\rangle\!\rangle_{cv}$ and $\C\langle\!\langle X\rangle\!\rangle_{cv}$, respectively.
\end{prop}

By straight-forward computation one finds the following crucial relation between the above two elements.
\begin{prop} One has
$$\calL_{cv} = \bfq(\calI_{cv}).$$
\end{prop}

\begin{rem}  \label{rem:bfq}
The map $\bfq$ is added due to the formula $$I_{[0, 1]}(\{0\}^{s_1-1}, \gs_1, \cdots, \{0\}^{s_r-1}\gs_r)= Li_{s_1,\cdots, s_r}(\gs_1, \gs_2\gs_1^{-1}, \cdots, \gs_r\gs_{r-1}^{-1}).$$
\end{rem}

Racinet proves the following fundamental result (see \cite[Corollary~2.4.4]{Racinet2002}) 
by Poincar\'{e}-Birkhoff-Witt Theorem.

\begin{thm}\label{thm:diagonalElement}
Any diagonal series $\Phi_{cv}$ of $(\C\langle\!\langle Y\rangle\!\rangle_{cv}, \Delta_{\ast})$ is the image of a diagonal element of $(\C\langle\!\langle Y \rangle\!\rangle, \Delta).$ Two such elements $\Phi_1$ and $\Phi_2$ are related by $$\Phi_2= \exp((\lambda_2-\lambda_1)y_{1,1})\Phi_1,$$ where $\lambda_1$ and $\lambda_2$ are the coefficients of $y_{1,1}$ in $\Phi_1$ and $\Phi_2$ respectively.
\end{thm}

With the above theorem available, he defines
\begin{defn}
Let $\calL$ (resp.\ $\calI$) be the unique diagonal element in $(\C\langle\!\langle Y\rangle\!\rangle, \Delta_{\ast})$ (resp.\ $(\C\langle\!\langle X \rangle\!\rangle, \Delta))$ such that $\pi_{cv}(\calL)= \calL_{cv}$ (resp.\ $\pi_{cv}(\calI)=\calI_{cv})$ and in which the coefficients of $y_{1,1}$ is 0 (resp.\ in which the coefficients of $x_0$ and $x_1$ are 0).
\end{defn}

\begin{rem}  
The coefficients of  $\calL$ (resp.\ $\calI$) may be viewed as regularized values of the series $Li_{s_1, \cdots, s_r}(\gs_1, \cdots, \gs_r)$ (resp.\ the iterated integral $I_{[0, 1]}(\gs_1, \cdots, \gs_r))$. Note that Racinet's choice is different from that given by Ihara, Kaneko, and Zagier in which the coefficient of $y_{1,1}$ is $T$ instead of 0.
\end{rem}

\begin{thm}\label{thm:relationCalLCalI}
The relation between $\calL$ and $\calI$  is given by
\begin{equation*}
\calL =\bfS(y_{1,1}) \cdot\bfq\pi_{Y}(\calI),
\end{equation*}
where $\bfS(u)= \displaystyle \exp{\left(\sum_{n=2}^{\infty}\frac{(-1)^{n-1}\zeta(n)}{n} u^n \right)}.$
\end{thm}

Let's recall the linear map defined in \cite{IKZ2006}:
\begin{align*}
\rho: \R[T]&\to \R[T],\\
 \rho(e^{Tu})&\mapsto A(u)e^{Tu},
\end{align*}
where $A(u)=\bfS(u)^{-1}.$
Let $\calL':= \exp(Ty_{1,1})\cdot\calL$ and $\calI':=\exp(Tx_1)\cdot\calI$.
By Theorem~\ref{thm:diagonalElement}, the coefficients of $y_{1,1}$ and $x_1$ in $\calL'$ 
and $\calI'$ are both equal to $T$. Note that $y_{1,1}= x_1.$ We denote by 
$Li^{\ast}_{s_1, \cdots, s_r}(\gs_1, \cdots, \gs_r)$ 
(resp.\ $I^{\Sha}_{[0, 1]}(\gs_1, \cdots, \gs_r))$ the coefficient of 
$y_{s_1, \gs_1} \cdots y_{s_r, \gs_r}$ (resp.\ $x_{\gs_1} \cdots x_{\gs_r})$ 
in $\calL'$ (resp.\ $\calI'$). Finally we show that the above relation in the sense of 
Racinet is equivalent to that given in \cite{IKZ2006}.

\begin{thm}
For any indices $\mathbf{s}=(s_1, \cdots, s_r)$ and $\mathbf{\gs}= (\gs_1, \cdots, \gs_r)$, we have
\begin{equation*}
\rho(Li^{\ast}_{s_1, \cdots, s_r}(\gs_1, \gs_2\gs_1^{-1}, \cdots, \gs_r\gs_{r-1}^{-1} )) =
I^{\Sha}_{[0, 1]}(\{0\}^{s_1-1}, \gs_1, \cdots, \{0\}^{s_r-1}, \gs_r) .
\end{equation*}
\end{thm}

\begin{proof}
Dividing the both sides of the equation in Theorem~\ref{thm:relationCalLCalI} by $\bfS(y_{1,1})$, 
we obtain 
\begin{align*}
A(y_{1,1})\cdot\calL'=\exp(Ty_{1,1})\cdot A(y_{1,1})\cdot \calL=&\exp(Ty_{1,1})\cdot \bfq \pi_{Y}(\calI)\\
=&\bfq \pi_{Y}(\exp(Tx_1)\calI)= \bfq \pi_{Y}(\calI').
\end{align*}
Hence 
\begin{equation*}
\rho(\calL')= \rho(\exp(Ty_{1,1})\cdot\calL)= A(y_{1,1})\cdot\exp(Ty_{1,1})\cdot\calL=\bfq \pi_{Y}(\calI'). \end{equation*}
The statement follows by comparing the coefficient of 
$y_{s_1,\gs_1}y_{s_1,\gs_2\gs_1^{-1}} \ldots y_{s_r,\gs_r\gs_{r-1}^{-1}} $ on
both sides of the equation.
\end{proof}

\section{Some generating functions}
For all $\ga\in\Z/N\Z$ put
\begin{equation*}
\frac{1}{\eta^{\ga} e^x-1}=\frac{\gd_{\ga,0}}x+\sum_{n=0}^{\infty} \frac{\om^{N}_{n;\ga} }{n!} x^n.
\end{equation*}
In this section, we will
provide explicit expressions of $\om$'s at lower levels. In the next few lemmas
we first present some facts that will be useful in our future computations.
\begin{lem}\label{lem:oppositega}
For all $n\ge 0$ and for all $N$ we have
\begin{equation*}
            \om^{N}_{n;0}=\frac{B_{n+1}}{n+1}.
\end{equation*}
For all $n\in\N$ and $\ga\in\Z/N\Z$ we have
\begin{equation}\label{equ:om(-1)^nom}
    \om^{N}_{n;\ga}=-(-1)^n  \om^{N}_{n;-\ga}.
\end{equation}
\end{lem}
\begin{proof} The first equation is straightforward from the definition.
For the second, we have
\begin{equation*}
\sum_{n=0}^\infty \frac{ \om^{N}_{n;\ga}+(-1)^n  \om^{N}_{n;-\ga}}{n!} x^n=
\frac{1}{\eta^\ga e^x-1}+\frac{1}{\eta^{-\ga} e^{-x}-1}=\frac{1-\eta^\ga e^x}{\eta^\ga e^x-1}=-1.
\end{equation*}
The lemma follows immediately.
\end{proof}

\begin{lem}  \label{lem:omNgeneral}
For all $n\in\N$ and $\ga\in\Z/N\Z$
 $\om^N_{2n-1;\ga}$ is real and $\om^N_{2n;\ga}$ is pure imaginary.
They both lie in the cyclotomic field $\Q_N$. Further
\begin{equation}\label{equ:omNgeneral}
\sum_{\ga=1}^{N-1} \om^{N}_{n;\ga}=\frac{(N^{n+1}-1)B_{n+1}}{n+1}.
\end{equation}
\end{lem}
\begin{proof}
Set
\begin{equation*}
    f_\ga(x)=\frac{1}{\eta^\ga e^x-1}.
\end{equation*}
Then
\begin{equation*}
    f_\ga(x)+\ol{f_\ga}(-x)=f_\ga(x)+f_{N-\ga}(-x)=\frac{1}{\eta^\ga e^x-1}-\frac{\eta^\ga e^x}{\eta^\ga e^x-1}=-1.
\end{equation*}
where $\ol{f_1}$ is the complex conjugation of $f_1$. Thus
$$\om^{N}_{2n-1;\ga}-\ol{\om^{N}_{2n-1;\ga}}=\om^{N}_{2n;\ga}+\ol{\om^{N}_{2n;\ga}}=0 \quad\forall n\ge 1.$$
This proves the first sentence of the lemma.
Observe that $F(t):=t^N-1=\prod_{j=0}^{N-1}(t-\eta^j)$. By the Leibniz rule
\begin{equation*}
 \sum_{j=0}^{N-1} \frac{t}{t-\eta^j }=\frac{tF'(t)}{F(t)}=\frac{Nt^N}{t^N-1}.
\end{equation*}
Using the substitution $t=e^{-x}$ we get
\begin{equation}\label{equ:Ngeneral}
\sum_{j=0}^{N-1} \frac{1}{\eta^j e^x-1}=  \frac{N}{e^{Nx}-1}.
\end{equation}
Now \eqref{equ:omNgeneral} quickly follows from this.
\end{proof}

\begin{cor}  \label{cor:omN=2}
Let $N=2$. Then for all $n\ge 0$ we have
\begin{equation}\label{equ:N=2ga=1}
\om^2_{n;1} =  \frac{(2^{n+1}-1) B_{n+1} }{n+1}.
\end{equation}
\end{cor}
\begin{proof}
This directly follows from \eqref{equ:omNgeneral} when $N=2$.
\end{proof}

\begin{lem}  \label{lem:omN=3}
Let $N=3$ and $\eta=e^{2\pi \sqrt{-1}/3}$. Then $\om^3_{0;1}=(\eta^2-1)/3$ and for all $n\ge 1$
\begin{align}
 \om^3_{2n-1;1}=\om^3_{2n-1;2}=&\, \frac{(3^{2n}-1) B_{2n} }{4n}, \label{equ:N=3ga=1odd}\\
 \om^3_{2n;1}=-\om^3_{2n;2}=&\,-\frac{\sqrt{-3}}{6(2n+1)} \sum_{j=0}^{2n} 3^{2n-j}(2^{j+1}-1) \binom{2n+1}{2n-j}B_{2n-j}.   \label{equ:N=3ga=1even}
\end{align}
\end{lem}
\begin{proof}
Clearly
 $$\om^3_{0;1}=\frac{1}{\eta-1}=\frac{\eta^2-1}{3}.$$
By \lemref{lem:omNgeneral} we see that $\om^3_{2n-1;1}$ are real and
$\om^3_{2n;1}$ are pure imaginary for all $n\ge 1.$
Further, by \eqref{equ:omNgeneral}
\begin{equation*}
\om^{N}_{n;1}+\ol{\om^{N}_{n;1}}=\om^{N}_{n;1}+\om^{N}_{n;2}=\frac{(3^{n+1}-1)B_{n+1}}{n+1}.
\end{equation*}
Hence we have
$$2 \om^3_{2n-1;1} =\frac{(3^{2n}-1) B_{2n} }{2n}$$
which quickly yields \eqref{equ:N=3ga=1odd}. On the other hand
\begin{align*}
 f_1(x)-f_2(x)
&= -\frac{\sqrt{-3}e^x}{e^{2x}+e^x+1}= -\frac{\sqrt{-3}(e^{2x}-e^{x})}{e^{3x}-1}\\
&= -\frac{\sqrt{-3}}{3} \sum_{n=0}^{\infty} \frac{2^{n+1}-1}{(n+1)!} x^n \cdot \sum_{m=0}^{\infty} \frac{3^mB_m}{m!}x^m.
\end{align*}
Hence \eqref{equ:N=3ga=1even} follows immediately.
\end{proof}

\begin{lem}  \label{lem:omN=4}
Let $N=4$ and $n\ge 0$. Then $\om^4_{n;2}= \om^{2}_{n;1}$ is given by \eqref{equ:N=2ga=1}.
Further we have  $\om^4_{n;3}=\ol{\om^4_{n;1}}$ where
\begin{equation}\label{equ:N=4ga=1}
\om^4_{n;1} =
\left\{
  \begin{array}{ll}
  {\displaystyle \frac{-1-\sqrt{-1}}2, } & \hbox{if $n=0$;} \\
 {\displaystyle \frac{2^n(2^{n+1}-1) B_{n+1} }{n+1}, \phantom{\frac{\Big|}{\Big|}}} & \hbox{if $n$ is odd;} \\
    {\displaystyle -\frac{\sqrt{-1}}{2} E_n,} & \hbox{if $n\ge 2$ is even,}
  \end{array}
\right.
\end{equation}
where $E_n$ are Euler numbers defined by the generating function
\begin{equation*}
 {\rm sech}\, x=\frac{2}{e^{x}+e^{-x}}=\sum_{n=0}^\infty \frac{E_n}{n!} x^n.
\end{equation*}
\end{lem}
\begin{proof}
The relation $\om^4_{n;2}= \om^{2}_{n;1}$ is trivial. Let $\eta=\sqrt{-1}$. Clearly
 $$\om^4_{0;1}=\frac{1}{\eta-1}= \frac{-1-\sqrt{-1}}2.$$
By \lemref{lem:omNgeneral} we see that $\om^4_{2n-1;1}$ are real and $\om^4_{2n;1}$
are pure imaginary for all $n\ge 1.$ Further
\begin{multline*}
   f_1(x)+\ol{f_1}(x) = f_1(x)+f_3(x)
=\frac{-2}{e^{2x}+1}\\
=\frac{4}{e^{4x}-1}-\frac{2}{e^{2x}-1}
=\sum_{n=0}^\infty \frac{2^{n+1}(2^{n+1}-1)B_{n+1}}{(n+1)!} x^n.
\end{multline*}
This proves the odd $n$ case in \eqref{equ:N=4ga=1}. On the other hand
\begin{equation*}
 f_1(x)-\ol{f_1}(x) =  f_1(x)-f_3(x) =\frac{-2\eta e^{x}}{e^{2x}+1}
   =\frac{-2\eta}{e^{x}+e^{-x}}=-\eta\sum_{n=0}^\infty \frac{E_n}{n!} x^n.
\end{equation*}
This proves
the even $n$ case in \eqref{equ:N=4ga=1}.
\end{proof}

\section{Multiple divisor functions}\label{sec:MDFs}
Recall that $\eta=\exp(2\pi \sqrt{-1}/N)$ is the $N$-th root of unity. Recall
$\Q_N=\Q(\eta_N)$ is the $N$-th cyclotomic field.
As a generalization of the classical divisor sums we define for integers
$s_1,\dots,s_d\ge 0$ and $\ga_1,\dots,\ga_d\in\Z/N\Z$
the \emph{multiple divisor sum at level $N$} by
\begin{equation} \label{def:sigma}
\gs_{s_1,\dots,s_d}^{\ga_1,\dots,\ga_d}(n) = \sum_{\substack{u_1 v_1 + \cdots + u_d v_d = n\\
u_1 > \dots > u_d>0}} \eta^{\ga_1v_1 + \cdots+ \ga_d v_d} v_1^{s_1} \dots v_d^{s_d}.
\end{equation}
In general, these sums are complex numbers lying in $\Q_N$.

Suppose $\bfs=(s_1,\dots,s_d)\in\N^d$ and $\bfga=(\ga_1,\dots,\ga_d)\in(\Z/N\Z)^d$. We
set $\bfs-\bfone=(s_1-1,\dots,s_d-1)$. Then the \emph{multiple divisor function} at level $N$
is the generating $q$-series of the multiple divisor sum $\gs_{\bfs}^{\bfga}$
given by
\begin{equation*}
 [\bfs;\bfga]_N(q) := \frac{1}{(s_1-1)! \dots (s_d-1)!} \
\sum_{n>0} \gs_{\bfs-\bfone}^{\bfga}(n) q^n \,\,\in \Q_N[\![q]\!].
\end{equation*}
Here and in what follows we will simply write $[\bfs;\bfga]_N$ or even $[\bfs;\bfga]$ instead of
$[\bfs;\bfga]_N(q)$. Similar to MZVs we call
$|\bfs|:=s_1+\dots+s_d$ the \emph{weight} and $\dep(\bfs):=d$ the \emph{depth}.
At level $N=2$ we use the special notation by putting the
letter o (for ``odd'') on top of $s_j$ if and only if the corresponding color $\ga_j=1$ is odd.

\begin{eg} We give a few examples at level $N=2$:
\begin{align*}
[\baro{2}]& =[2;1]_2 = -q+q^2-4q^3+5q^4-6q^5+4q^6-8q^7+13q^8+ \dots, \\
[2,\baro{1}] & =[2,1;0,1]_2 =-q^3-4q^5+q^6-9q^7+4q^8-17q^9+8q^{10}-25q^{11} ,
\end{align*}
\end{eg}

Recall that in \cite{BachmannKu2013} a \emph{normalized polylogarithm} is defined by
\begin{equation*}
\tLi_s(z) := \frac{ \Li_{1-s}(z) }{\gG(s)},
\end{equation*}
where for $s\in\N$, $z\in\C$, $|z|<1$ the polylogarithm $\Li_s(z)=\sum_{n\ge 1} \frac{z^n}{n^s}$
of weight $s$. By \cite{Foata2010} we see that $\tLi_s(z)$ is a rational function in
$z$ and is holomorphic in $|z|<1$.

\begin{prop}\label{prop:ExpressMDFbyLi}
For $q \in \C$ with   $|q|<1$ and for all $s_1,\dots,s_d \in \N$
we can write the MDFs as
\begin{equation*} [s_1,\dots ,s_d;\ga_1,\dots,\ga_d] = \sum_{n_1 > \dots > n_d>0} \tLi_{s_1}\left(\eta^{\ga_1}q^{n_1}\right) \cdots \tLi_{s_d}\left(\eta^{\ga_d} q^{n_d}\right). \end{equation*}
\end{prop}

\begin{proof}
We leave the straightforward computation to the interested reader.
\end{proof}

\begin{lem}\label{lem:productLiLi}
For $\ga, \gb\in \bbZ/N\bbZ$,
\begin{equation*}
\tLi_{a}(\eta^{\ga} z) \cdot  \tLi_{b}(\eta^{\gb} z) =
\sum_{j=1}^a \gl^{j;N}_{a,b;\ga-\gb} \tLi_{j}(\eta^{\ga}z)
   +\sum_{j=1}^b \gl^{j;N}_{b,a;\gb-\ga}  \tLi_{j}(\eta^{\gb} z)
   +\gd_{\ga,\gb} \tLi_{a+b}(\eta^{\ga} z)
\end{equation*}
where the coefficients $\gl^{j;N}_{a,b;\ga} \in \Q(\eta)$ are given by
\begin{equation}\label{equ:gl=om}
\gl^{j;N}_{a,b;\ga} = (-1)^{b-1} \binom{a+b-j-1}{a-j}
 \frac{\om^N_{a+b-j-1;\ga} }{(a+b-j-1)!}.
\end{equation}
It satisfies that
\begin{equation}\label{equ:gl(-1)^jgl}
\gl^{j;N}_{a,b;\ga} = (-1)^{a+b} \gl^{j;N}_{b,a;\ga}
= (-1)^{a+b+j} \gl^{j;N}_{a,b;-\ga}= (-1)^{j} \gl^{j;N}_{b,a;-\ga}.
\end{equation}
\end{lem}
\begin{proof}
First we observe that \eqref{equ:gl(-1)^jgl} follows from \eqref{equ:om(-1)^nom}.

Now we consider the generating series
\begin{equation}\label{equ:Lgax}
L_{\ga}(x)= \sum_{k_1=1}^{\infty} \tLi_{k_1}(\eta^{\ga} z) x^{k_1-1}
\quad\text{and}\quad
L_{\gb}(y)= \sum_{k_2=1}^{\infty} \tLi_{k_2}(\eta^{\gb} z) y^{k_2-1}.
\end{equation}
A simple computation shows that
\begin{equation*}
L_{\ga}(x)\cdot L_{\gb}(y) =
\frac{1}{\eta^{\ga-\gb} e^{x-y}-1} L_{\ga}(x)
+ \frac{1}{\eta^{\gb-\ga} e^{y-x}-1} L_{\gb}(y).
\end{equation*}

We now consider two cases: (1) $\ga= \gb$ and (2) $\ga \neq \gb.$

(1). If $\ga= \gb,$ then the proof is exactly the same as that of \cite[Lemma~2.6]{BachmannKu2013}.

(2). If $\ga \neq \gb,$ then we have
\begin{align*}
 L_{\ga}(x) \cdot L_{\gb}(y) = &\, \frac{ e^x \eta^{\ga}z}{1- e^x \eta^{\gb}z}\cdot \frac{ e^y \eta^{\ga}z}{1-e^y \eta^{\gb}z}\\
 = &\, \frac{1}{\eta^{\ga-\gb}e^{x-y}-1} L_{\ga}(x) +\frac{1}{\eta^{\gb-\ga}e^{y-x}-1} L_{\gb}(y)  \\
 = &\, \sum_{n>0} \frac{\om^N_{n;\ga-\gb}}{n!}(x-y)^n L_{\ga}(x)
+\sum_{n>0} \frac{\om^N_{n;\gb-\ga}}{n!}(y-x)^n L_{\gb}(y).
 \end{align*}
By definition \eqref{equ:Lgax} this equals
 \begin{align*}
  & \sum_{n>0} \frac{\om^N_{n;\ga-\gb}}{n!}(x-y)^n \sum_{j\ge 1} \tLi_{j}(\eta^{\ga}z) x^{j-1}
 +\sum_{n>0} \frac{\om^N_{n;\gb-\ga}}{n!}(y-x)^n \sum_{j\ge 1} \tLi_{j}(\eta^{\gb}z) y^{j-1}\\
=&\,  \sum_{b\ge 1}  \sum_{n=1}^{b-1} \sum_{j\ge 1}
 \frac{\om^N_{n;\ga-\gb}}{n!} \binom{n}{b-1}
 (-1)^{b-1}  x^{n-b+j} y^{b-1}  \tLi_{j}(\eta^{\ga}z)  \\
+ &\, \sum_{a\ge 1}  \sum_{n=1}^{a-1} \sum_{j\ge 1}
\frac{\om^N_{n;\gb-\ga}}{n!} \binom{n}{a-1} (-1)^{a-1}  x^{a-1} y^{n-a+j} \tLi_{j}(\eta^{\gb}z)\\
=&\,   \sum_{a\ge 1}\sum_{b\ge 1}  \sum_{j=1}^a
\frac{\om^N_{a+b-j-1;\ga-\gb}}{(a+b-j-1)!} \binom{a+b-j-1}{b-1}
 (-1)^{b-1} \tLi_{j}(\eta^{\ga}z)x^{a-1} y^{b-1} \\
+&\, \sum_{a\ge 1}\sum_{b\ge 1}  \sum_{j=1}^b
\frac{\om^N_{a+b-j-1;\gb-\ga}}{(a+b-j-1)!} \binom{a+b-j-1}{a-1}
(-1)^{a-1}\tLi_{j}(\eta^{\gb}z)x^{a-1} y^{b-1}.
\end{align*}
The lemma now follows from a comparison of coefficients.
\end{proof}

\begin{cor}  \label{lem:N=2}
Let $N=2$. For $a,b \in \N$ and $\ga, \gb\in \Z/2\Z$  we have
\begin{equation*}
\tLi_{a}(\eta^{\ga} z) \cdot  \tLi_{b}(\eta^{\gb} z) =
\sum_{j=1}^a \gl^{j;2}_{a,b;\ga-\gb} \tLi_{j}(\eta^{\ga}z)
   +\sum_{j=1}^b \gl^{j;2}_{b,a;\gb-\ga}  \tLi_{j}(\eta^{\gb} z)
   +\gd_{\ga,\gb} \tLi_{a+b}(\eta^{\ga} z)
\end{equation*}
where the coefficients $\gl^{j;2}_{a,b;\ga} \in \Q(\eta)$ are given by
\begin{align*}
\gl^{j;2}_{a,b;0} =&\, (-1)^{b-1} \binom{a+b-j-1}{a-j} \frac{B_{a+b-j}}{(a+b-j)!}, \\
\gl^{j;2}_{a,b;1} =&\,  \big(2^{a+b-j}-1\big) \gl^{j;2}_{a,b;0}.
\end{align*}
\end{cor}
\begin{proof}
This follows from \corref{cor:omN=2} easily.
\end{proof}

\begin{eg}
Let $N=2$. We have $\gl^{1;2}_{1,1;0} = \gl^{1;2}_{1,1;1} = B_1 = -\frac{1}{2}$ and thus
\begin{equation*}
\tLi_1(z) \cdot  \tLi_1(z) =\tLi_2(z)-\tLi_1(z),  \quad
\tLi_1(z) \cdot  \tLi_1(-z) =-\frac12\tLi_1(z) -\frac12 \tLi_2(-z).
\end{equation*}
Therefore the product  $[1;0]\cdot[1;1]$ is given by \eqref{equ:1xodd1}.
\end{eg}

\begin{prop}\label{prop:l2-expli}
For $\ga, \gb \in \bbZ/N\bbZ$, we have
\begin{multline*}
[s;\ga] \cdot [t;\gb] =
[s, t;\ga, \gb]+ [t, s;\gb, \ga]+\gd_{\ga,\gb} [s+t;\ga]\\
 +\sum_{j=1}^{s}\gl^{j;N}_{s, t;\ga-\gb} [j;\ga]
 +\sum_{j=1}^{t}\gl^{j;N}_{t, s;\gb-\ga} [j;\gb].
\end{multline*}
\end{prop}
\begin{proof}
By \propref{prop:ExpressMDFbyLi} we have
\begin{align*}
[s;\ga]\cdot [t;\gb] =& \left(\sum_{n_1>n_2>0}  + \sum_{n_2>n_1>0}
+ \sum_{n_1=n_2} \right) \tLi_{s}(\eta^{\ga n_1} q^{n_1}) \tLi_{t}(\eta^{\gb n_2}q^{n_2})\\
=&\, [s, t;\ga, \gb]+ [t, s;\gb, \ga]
+ \sum_{n=1}^\infty \tLi_{s}(\eta^{\ga n} q^{n})\tLi_{t}(\eta^{\gb n}q^{n}).
\end{align*}
The proposition now follows from \lemref{lem:productLiLi}.
\end{proof}

\begin{eg}
The first nontrivial product at any level $N$ is given by
\begin{equation}\label{equ:[1;1][1;1]}
[1;1]\cdot [1;1]=2[1,1;1,1]+[2;1]-[1;1]
\end{equation}
since $\gl^{1;N}_{1,1;0}=B_1=-\frac{1}{2}$ for all $N$.
\end{eg}

\begin{eg}
The first products of MDFs at level $N=2$ are given by
{\allowdisplaybreaks
\begin{align}
[1]\cdot [\baro{1}]&\, =[1,\baro{1}]+[\baro{1},1]-\frac12[1]-\frac12[\baro{1}], \label{equ:1xodd1}\\
[\baro{2}] \cdot [\baro{1}]&\, =[2,\baro{1}]+[1,\baro{2}]-\frac12[\baro{2}]+[\baro{3}],\label{equ:odd2x1}\\
[2]\cdot [\baro{1}]&\, =[2,\baro{1}]+[\baro{1},2]+\frac14\Big([1]-2[2]-[\baro{1}]\Big), \label{equ:2xodd1}\\
[2]\cdot [\baro{2}]&\, =[2,\baro{2}]+[\baro{2},2]-\frac14\Big([2]+[\baro{2}]\Big).\label{equ:2xodd2}
\end{align}}
Here we have used the letter o (for ``odd'') on top of $s_j$ if the corresponding color $\ga_j=1$ is odd.
\end{eg}

\begin{eg}
At level $N=3$ we have
{\allowdisplaybreaks
\begin{align}
[1;1]\cdot[1;2]=&\,[1,1;1,2]+[1,1;2,1]+\frac{\sqrt{-3}}{6}\Big([1;2]-[1;1]\Big), \notag\\
[1;1]\cdot[2;0]=&\,[2,1;0,1]+[1,2;1,0]+\frac{1}{2}\Big([1;1]-[1;0]\Big)-\frac{\sqrt{-3}}{6}[2;0],\label{equ:N=3[2,0]}\\
[2;1]\cdot[1;2]=&\,[2,1;1,2]+[1,2;2,1]+\frac{1}{2}\Big([1;2]-[1;1]\Big)-\frac{\sqrt{-3}}{6}[2;1],\notag\\
[2;1]\cdot[2;2]=&\,[2,2;1,2]+[2,2;2,1]+\frac{1}{2}\Big([2;1]+[2;2]\Big)+\frac{\sqrt{-3}}{9}\Big([1;2]-[1;1]\Big),\notag\\
[3;1]\cdot[2;2]=&\,[3,2;1,2]+[2,3;2,1]+\frac{1}{2}[3;1]-\frac{\sqrt{-3}}{9}[2;1]-\frac{\sqrt{-3}}{18}[2;2].\notag
\end{align}}
\end{eg}

\begin{eg}
At level $N=4$ we have
{\allowdisplaybreaks
\begin{align}
[1;1]\cdot[1;2]=&\,[1,1;1,2]+[1,1;2,1]+\frac{\sqrt{-1}}{2}\Big([1;1]-[1;2]\Big),\notag\\
[1;1]\cdot[1;3]=&\,[1,1;1,3]+[1,1;3,1]-\frac{1}{2}\Big([1;1]+[1;3]\Big),\notag\\
[1;1]\cdot[2;0]=&\,[2,1;0,1]+[1,2;1,0]+\frac{1}{2}\Big([1;0]-[1;1]+\sqrt{-1}[2;0]\Big),\label{equ:N=4[2,0]}\\
[1;1]\cdot[2;2]=&\,[1,2;1,2]+[2,1;2,1]+\frac{1}{2}\Big([1;2]-[1;1]-\sqrt{-1}[2;2]\Big),\notag\\
[1;1]\cdot[2;3]=&\,[1,2;1,3]+[2,1;3,1]+\frac{1}{4}\Big([1;3]-[1;1]-2[2;3]\Big),\notag\\
[3;1]\cdot[2;2]=&\,[3,2;1,2]+[2,3;2,1]+\frac{1}{2}\Big([1;1]-[3;1]-[1,2]+\sqrt{-1}[2;1]\Big)+
   \frac{\sqrt{-1}}{4}[2, 2].\notag
\end{align}}
\end{eg}

\section{The algebra of MDFs}
In this section we will extend \propref{prop:l2-expli} to the general case by
studying the underlying algebra structure of the MDFs.

\begin{defn}
We define $\MD_N$ to be the $\Q_N$-vector space generated by
$[\emptyset]=1 \in \Q_N[\![q]\!]$ and all MDFs $[\bfs;\bfga]$ for
$\bfs\in\N^d$ and $\bfga\in(\Z/N\Z)^d$.
On $\MD_N$ we have the increasing filtration $\filw_{\bullet}$
given by the weight and the increasing filtration $\fille_{\bullet}$
given by the depth, i.e.,
\begin{align*}
\filw_k(\MD_N) &:=  \big\langle[\bfs;\bfga] \,\big|\, \bfs\in\N^d, \bfga\in(\Z/N\Z)^d, d\in\N, |\bfs|\le k \rangle_{\Q_N}, \\
\fille_d(\MD_N) &:= \big\langle[\bfs;\bfga] \,\big|\, \bfs\in\N^r, \bfga\in(\Z/N\Z)^r, r\le d \rangle_{\Q_N}.
\end{align*}
When considering the depth and weight filtrations at the same time we write
$\filwle_{k,d} := \filw_k \fille_d$.
As usual we define the graded pieces by
\begin{align*}
\grw_k(\MD_N) &:=   \filw_k(\MD_N) \slash \filw_{k-1}(\MD_N) \\
\grl_d(\MD_N) &:=  \fille_d(\MD_N) \slash \fille_{d-1}(\MD_N).
\end{align*}
and $\grwl_{k,d} := \grw_k \grl_d$.
\end{defn}

\begin{defn}\label{defn:qMZN}
Let $G_N={\rm Gal}(\Q_N/\Q)$. Let $\qMZ_N$ be the $\Q_N$-vector space generated by
\begin{equation}\label{equ:SN}
S_N:=\Big\{[\bfs;\bfga]: s_1>1 \Big\}\cup \Big\{g_\gb(\bfs;\bfga):
     1\le \gb<N,\bfs\in\N^d,\bfga\in(\Z/N\Z)^d,\ d\ge 0\Big\},
\end{equation}
where $g_\gb(\bfs;\bfga):=[1,\bfs;0,\bfga]-[1,\bfs;\gb,\bfga]$ is defined precisely
as follows. If $\gb<\ga_1$ or $\bfs=\emptyset$ then
\begin{multline}\label{equ:defng_gb1}
g_\gb(\bfs;\bfga):=\sum_{n>n_1 > \dots > n_d>0} \Big(\tLi_1\big(q^n\big) -\tLi_1\big(\eta^{\gb} q^{n+\gb}\big ) \Big) \tLi_{s_1}\left(\eta^{\ga_1}q^{n_1}\right) \cdots \tLi_{s_d}\left(\eta^{\ga_d} q^{n_d}\right)\\
-\sum_{\ell=1}^\gb\sum_{n_1>\dots>n_d>0} \tLi_1\big(\eta^{\gb} q^{n_1+\ell}\big ) \Big) \tLi_{s_1}\left(\eta^{\ga_1}q^{n_1}\right) \cdots \tLi_{s_d}\left(\eta^{\ga_d} q^{n_d}\right),
\end{multline}
where $n_1=0$ if $\bfs=\emptyset$.
If $d\ge1$ and $\gb\ge \ga_1$ then
\begin{multline}\label{equ:defng_gb2}
g_\gb(\bfs;\bfga):=\sum_{n>n_1 > \dots > n_d>0} \Big(\tLi_1\big(q^n\big) -\tLi_1\big(\eta^{\gb} q^{n+\gb-N}\big ) \Big) \tLi_{s_1}\left(\eta^{\ga_1}q^{n_1}\right) \cdots \tLi_{s_d}\left(\eta^{\ga_d} q^{n_d}\right)\\
+\sum_{\ell=\gb+1}^N \sum_{n_1>\dots>n_d>0} \tLi_1\big(\eta^{\gb} q^{n_1-N+\ell}\big ) \Big) \tLi_{s_1}\left(\eta^{\ga_1}q^{n_1}\right) \cdots \tLi_{s_d}\left(\eta^{\ga_d} q^{n_d}\right).
\end{multline}
Let $\qMZ_N(\Q)$ be the $\Q$-vector space generated by $S_N$ given by \eqref{equ:SN}.
Define
\begin{align*}
\ol{\qMZ_N} &\, =\oplus_k \gr_k^{\operatorname{W};N}  \qMZ_N ,\\
 \overline{\qMZ_N}(\Q)
&\, = \oplus_k \big(\gr_k^{\operatorname{W};N}  \qMZ_N \big)^{G_N}
=\oplus_k \gr_k^{\operatorname{W};N}  \qMZ_N(\Q).
\end{align*}
Further we define $\MZ_N$ to be the $\Q_N$-vector space generated by
\begin{equation*}
\Big\{\zeta_N(\bfs;\bfga): s_1>1 \Big\}\cup \Big\{\gG_\gb(\bfs;\bfga):
     1\le \gb<N,\bfs\in\N^d,\bfga\in(\Z/N\Z)^d,\ d\ge 0\Big\},
\end{equation*}
where $\gG_\gb(\bfs;\bfga)= \zeta^\ast_N(1,\bfs;0,\bfga)-\zeta^\ast_N(1,\bfs;j,\bfga)$
are define by \eqref{equ:gG_gb} in section~\ref{sec:regularizationMZV}.
Similarly, $\MZ_N(\Q)$ is the corresponding $\Q$-vector space.
\end{defn}

\propref{prop:l2-expli} is a special case of \thmref{thm:md-algebra} given below.
To prove the theorem in general we now consider the associated quasi-shuffle
algebras (cf.\ \cite{HoffmanIh2012}). Let $\F$ be a field of characteristic $0$
(we will take $\F$ to be either $\Q$ or $\Q_N$ later). Let
$$A=\left\{ z_{j;\ga}: j\in \N,\ga\in\Z/N\Z \right\}$$
be the alphabet, $\F A$ the $\F$-vector space generated by letters in $A$ and
$\F\langle A \rangle$ the noncommutative polynomial algebra over $\F$
generated by words with letters in $A$. For a commutative and associative product $\diamond$ on
$\F A$, $a,b \in A$ and $w,v \in \F\langle A\rangle$ we define on $\F\langle A\rangle$
recursively a product by $1*w=w*1=w$ and
\begin{equation}\label{defn:stuffle}
aw \ast bv := a(w \ast bv) + b(aw \ast v) + (a \diamond b)(w \ast v).
\end{equation}
Equipped with this product one has the
\begin{prop}\label{prop:HI-quasi}
The vector space $\F\langle A\rangle$ with the product $\ast$ is a commutative $\F$-algebra.
\end{prop}
\begin{proof}
See \cite[Theorem 2.1]{HoffmanIh2012}.
\end{proof}

Motivated by the product expression of the polylogarithms in \lemref{lem:productLiLi}
we define the product $\diamond$ on $\F A$ by
\begin{equation}\label{equ:diamondDefn}
z_{a;\ga} \diamond z_{b;\gb} =
 \sum_{j=1}^a \gl^{j;N}_{a,b;\ga-\gb} z_{j;\ga}
        + \sum_{j=1}^b \gl^{j;N}_{b,a;\gb-\ga} z_{j;\ga} + \gd_{\ga,\gb} z_{a+b;\ga}.
\end{equation}
This is an commutative and associative product  on $\F A$, because it arises from the product of the pairwise linearly independent polylogarithms $\tLi_t(\eta^\ga z)$ in \lemref{lem:productLiLi},
and therefore   $\left( \F\langle A\rangle, \ast \right)$ is a commutative $\F$-algebra
by \propref{prop:HI-quasi} above.
\thmref{thm:md-algebra} now follows from the next proposition.

\begin{prop}\label{prop:hoffmannalgebra}
Let $[\ \ ]: (\Q_N\langle A\rangle, *) \longrightarrow (\MD_N, \cdot)$ be the $\Q_N$-linear map
such that $[z_{s_1;\ga_1}\dots z_{s_d;\ga_d}]:=[s_1,\dots,s_d;\ga_1,\dots,\ga_d]$. Then we have
\begin{equation*}
[ w \ast v ]    = [w] \cdot [v] \quad \forall w,b\in \Q_N\langle A\rangle.
\end{equation*}
Therefore $(\MD_N, \cdot)$ is a $\Q_N$-algebra and $[\ \ ]$ a $\Q_N$-algebra homomorphism.
\end{prop}
\begin{proof}
This can be proved by the same argument as in the MZV case
(see, for e.g., \cite[Theorem~3.2]{Hoffman1997}) by using induction on the
depth of the words $w$ and $v$ together with \lemref{lem:productLiLi}.
\end{proof}

\begin{thm} \label{thm:md-algebra}
The $\Q_N$-vector space $\MD_N$
has the structure of a bifiltered $\Q_N$-Algebra $(\MD_N,\, \cdot,\,\filw_{\bullet},\,\fille_{\bullet})$,
where the multiplication is the natural multiplication of formal power series and
the filtrations $\filw_{\bullet}$ and  $\fille_{\bullet}$ are induced by the $weight$
and $depth$,  in particular
\begin{equation*}
\filwle_{k_1,d_1}(\MD_N) \cdot  \filwle_{k_2,d_2}(\MD_N) \subset \filwle_{k_1+k_2,d_1+d_2}(\MD_N).
\end{equation*}
\end{thm}
\begin{proof}
This follows easily from \propref{prop:hoffmannalgebra},
the definition in \eqref{defn:stuffle} and \eqref{equ:diamondDefn}.
\end{proof}

\begin{eg} \label{eg:stuffle[a;ga][b,c;gb,gam]}
For $a,b,c \in \N$ and $\ga,\gb,\gam\in\Z/N\Z$ we have
\begin{align*}
[a;\ga]\cdot[b,c;\gb,\gam] &=
[z_{a;\ga}* (z_{b;\gb} z_{c;\gam})]
= [z_{a;\ga} z_{b;\gb} z_{c;\gam} + z_{b;\gb} z_{a;\ga}z_{c;\gam} +z_{b;\gb} z_{c;\gam}z_{a;\ga} ]\\
&\hskip2cm +[ z_{b;\gb} (z_{a;\ga} \diamond z_{c;\gam}) + (z_{a;\ga} \diamond z_{b;\gb})z_{c;\gam} ]\\
& =[a,b,c;\ga,\gb,\gam]+[b,a,c;\gb,\ga,\gam]+[b,c,a;\gb,\gam,\ga]  \\
&\quad  + \gd_{\ga,\gam} [b, a+c;\gb,\ga] + \gd_{\ga,\gb}[a+b,c;\gam,\ga] \\
& \quad + \sum_{j=1 }^a \gl_{a,c;\ga-\gam}^j[b,j;\gb,\ga]
  + \sum_{j=1 }^c  \gl_{c,a;\gam-\ga}^j [b, j;\gb,\gam] \\
& \quad  + \sum_{j=1 }^a \gl_{a,b;\ga-\gb}^j [j, c;\ga,\gam]
+  \sum_{j= 1}^b \gl_{b,a;\gb-\ga}^j [j, c;\gb,\gam].
\end{align*}
\end{eg}

\begin{lem}\label{lem:diamond}
For all $a,b,s_1,\dots,s_d\in\N$ and $\ga,\gb,\ga_1,\dots,\ga_d\in\Z/N\Z$ we have
\begin{equation*}
[(z_{a;\ga} \diamond z_{b;\gb})z_{s_1;\ga_1}\dots z_{s_d;\ga_d} ] \in \qMZ_N.
\end{equation*}
\end{lem}
\begin{proof}
By the definition \eqref{equ:diamondDefn}
\begin{equation*}
z_{a;\ga} \diamond z_{b;\gb} =
 \sum_{j=1}^a \gl^{j;N}_{a,b;\ga-\gb} z_{j;\ga}
        + \sum_{j=1}^b \gl^{j;N}_{b,a;\gb-\ga} z_{j;\gb} + \gd_{\ga,\gb} z_{a+b;\ga}.
\end{equation*}
By the definition of $\gl$'s in \propref{prop:l2-expli} it is easy to see that
\begin{align*}
\gl^{1;N}_{a,b;\ga-\gb}z_{1;\ga} +\gl^{1;N}_{b,a;\gb-\ga} z_{1;\gb}
= &\, \frac{(-1)^{a-1}\om^N_{a+b-2;\ga-\gb}z_{1;\ga} + (-1)^{b-1}\om^N_{a+b-2;\gb-\ga}z_{1;\gb}}{(a-1)!(b-1)!} \\
= &\,
\frac{(-1)^{a-1}\om^N_{a+b-2;\ga-\gb}}{(a-1)!(b-1)!}\big(z_{1;\ga}-z_{1;\gb}\big).
\end{align*}
by \lemref{lem:oppositega}. Hence
$$[(z_{1;\ga}-z_{1;\gb})z_{s_1;\ga_1}\dots z_{s_d;\ga_d}]=g_\gb(\bfs;\bfga)-g_\ga(\bfs;\bfga) \in \qMZ_N.$$
This finishes the proof of the lemma.
\end{proof}

\begin{thm}\label{thm:qmz-subalgebra}
The vector space $\qMZ_N$ is a subalgebra of $\MD_N$.
\end{thm}
\begin{proof}
It suffices to show that $\qMZ_N$ is closed under multiplication.
We need to consider three cases:
\begin{enumerate}
\item[\upshape(i)] Let $f=[a,\dots;\ga,\dots]$ and $g=[b,\dots;\gb,\dots]$, $a>1$ and $b>1$,

\item[\upshape(ii)]  Let $f=[a,\dots;\ga,\dots]$ and $g=g_\gb(\bfs;\bfga)$, $a>1$ and $1\le \gb<N$,

\item[\upshape(iii)]  Let $f=g_\gam(\bfs;\bfga)$ and $g=g_\gb(\bft;\bfgb)$, $1\le \gb,\gam<N$.
\end{enumerate}

In case (i), by \propref{prop:hoffmannalgebra} we have
\begin{equation*}
f \cdot g = [z_{a;\ga} w] \cdot [z_{b;\gb} v] = [ z_{a;\ga} w \ast z_{b;\gb} v ]
\quad\text{for some }w,v \in \Q_N\langle A \rangle.
\end{equation*}
So in order to prove the statement we have to show that $[z_{a;\ga} w \ast z_{b;\gb} v]\in \qMZ_N$.
By the definition of the quasi-shuffle product $\ast$ we have
\begin{equation*}
z_{a;\ga} w \ast z_{b;\gb} v = z_{a;\ga}(w \ast z_{b;\gb} v)
+ z_{b;\gb}(z_{a;\ga} w \ast v) + (z_{a;\ga} \diamond z_{b;\gb})(w \ast v).
\end{equation*}
The first two terms obviously lie in $\qMZ_N$ after applying $[\ ]$ since $a,b>1$
and by \lemref{lem:diamond} $[(z_{a;\ga} \diamond z_{b;\gb})(w \ast v)] \in \qMZ_N$.
Thus in case (i) the multiplication is closed.

(ii) We assume for some $w,v \in \Q_N\langle A \rangle$
\begin{equation*}
f \cdot g = [z_{a;\ga} w] \cdot [(z_{1;0}-z_{1;\gb}) v]
= [ z_{a;\ga} w \ast z_{1;0} v ]-[ z_{a;\ga} w \ast z_{1;\gb} v ].
\end{equation*}
Then
\begin{multline*}
 z_{a;\ga} w \ast z_{1;0} v- z_{a;\ga} w \ast z_{1;\gb} v
= z_{a;\ga}(w \ast z_{1;0} v)-z_{a;\ga}(w \ast z_{1;\gb} v)\\
+ (z_{a;\ga} \diamond z_{1;0})(w \ast v)- (z_{a;\ga} \diamond z_{1;\gb})(w \ast v)
+ (z_{1;0}-z_{1;\gb})(z_{a;\ga} w \ast v).
\end{multline*}
By similar argument as in case (i) one can handle easily the first four terms on the right
hand side. Now assume $[z_{a;\ga} w \ast v]= \sum_{i=1}^r [\bfs_i;\bfa_i]$ we see that
$$[(z_{1;0}-z_{1;\gb})(z_{a;\ga} w \ast v) ] = \sum_{i=1}^r g_\gb(\bfs_i;\bfa_i)\in \qMZ_N.$$

(iii) Assuming $f=[(z_{1;0}-z_{1;\gb}) w]$ and $g=[(z_{1;0}-z_{1;\gam}) v]$ for
$w,v \in \Q\langle A \rangle$ we have
\begin{equation*}
f \cdot g = [ z_{1;0} w \ast z_{1;0} v ]-[ z_{1;0} w \ast z_{1;\gam} v ]
- [ z_{1;\gb} w \ast z_{1;0} v ]+[ z_{1;\gb} w \ast z_{1;\gam} v ].
\end{equation*}
Then
\begin{align*}
   &\,   z_{1;0} w \ast z_{1;0} v - z_{1;0} w \ast z_{1;\gam} v
- z_{1;\gb} w \ast z_{1;0} v + z_{1;\gb} w \ast z_{1;\gam} v\\
= &\, (z_{1;0}-z_{1;\gb})(w \ast z_{1;0} v)-(z_{1;0}-z_{1;\gb})(w \ast z_{1;\gam} v) \\
+&\,  (z_{1;0} \diamond z_{1;0})(w \ast v)- (z_{1;0} \diamond z_{1;\gam})(w \ast v) \\
+&\, (z_{1;0}-z_{1;\gam})(z_{1;0} w \ast v)- (z_{1;\gb} \diamond z_{1;0})(w \ast v) \\
+&\, (z_{1;\gb} \diamond z_{1;\gam})(w \ast v) -(z_{1;0}-z_{1;j})(z_{1;i} w \ast v) ).
\end{align*}
By the same argument as in (i), all the terms involving $\diamond$ operation lie in $\qMZ_N$
after applying $[ \ \ ]$. All the other terms have the form $g_\gb(\bfs;\bfga)$ or $g_\gam(\bfs;\bfga)$
after applying $[ \ \ ]$. This finishes the proof of the theorem.
\end{proof}

\begin{thm} \label{thm:polyadalg}
For any positive integer $N$ we set $t_N=\sum_{\ga=0}^{N-1} [1;\ga]$.
\begin{enumerate}
\item[\upshape(i)] We have  $\MD_N = \qMZ_N[t_N]$.
\item[\upshape(ii)]  The algebra $\MD_N$ is a polynomial ring over $\qMZ_N$ with indeterminate $t_N$, i.e. $\MD_N$ is isomorphic to $\qMZ_N[T]$ by sending $t_N$ to $T$.
\end{enumerate}
\end{thm}

\begin{proof}
(i) We will use the induction approach used by Bachmann and K\"uhn at level one.
But for the higher levels the appearance of different colors brings some complications
into our argument.

First we show that any $f \in \filw_{k}(\MD_N)$ can be written as a polynomial in $t_N$.
If we show that for a fixed $d$ and $f \in \filwle_{k,d}(\MD_N)$ one can find
$g_1 \in \filwle_{k,d}(\qMZ_N)$ and $g_2,g_3 \in \filwle_{k,d-1}(\MD_N)$ such that
$f$ can be written as
\begin{equation} \label{eq:polyrep}
f = g_1 + t_N\cdot g_2 + g_3,
\end{equation}
then the claim follows directly by induction on $d$.

To show \eqref{eq:polyrep} it suffices to assume
$f=[\{1\}^m, \bfs;\gb_1,\dots,\gb_m,\bfga]$, with $s_1 > 1$ and $k=m+|\bfs|$.
By induction on $m$ we now prove that every element of such form can be written
as in \eqref{eq:polyrep}. Notice we are using a double induction now with the outside layer
induction on $d$ and the inside on $m$.
If $m=0$ then clearly $f \in \filwle_{k,d}(\qMZ_N)$.
If $m=1$, $s_1>1$ and $\gb_1=0$ then we have
\begin{align*}
 &\,  -N[1,\bfs;0,\bfga]+\sum_{\gam=1}^{N-1} g_\gam(\bfs;\bfga)+t_N\cdot[\bfs;\bfga] \\
=&\, \sum_{\gam=0}^{N-1}\Big([1;\gam]\cdot [\bfs;\bfga]- [1,\bfs;\gam,\bfga]\Big)\\
=&\, \sum_{\gam=0}^{N-1}\Big[z_{1;\gam}*(z_{s_1;\ga_1}\dots z_{s_n;\ga_n})- z_{1;\gam}z_{s_1;\ga_1}\dots z_{s_n;\ga_n}\Big] \quad (n=d-1)\\
=&\, \sum_{\gam=0}^{N-1}\Big[z_{s_1;\ga_1}\big(z_{1;\gam}*(z_{s_2;\ga_2}\dots z_{s_n;\ga_n})\big)
    -(z_{1;\gam}\diamond z_{s_1;\ga_1}) (z_{s_2;\ga_2}\dots z_{s_n;\ga_n})\Big]\in \qMZ_N
\end{align*}
by \lemref{lem:diamond}. Hence $[1,\bfs;0,\bfga]\in \qMZ_N[t_N]$ and therefore for all $\gb\in\Z/N\Z$
\begin{equation*}
[1,\bfs;\gb,\bfga]=[1,\bfs;0,\bfga]-g_\gb(\bfs;\bfga)\in\qMZ_N[t_N].
\end{equation*}

Now we assume $m\ge 1$ and \eqref{eq:polyrep} holds for $f=[\{1\}^m,\bfs;\bfgb',\bfga]$
for all $\bfs\in \N^{d-m}$ with $s_1>1$ and all $\bfgb'=(\gb_1,\dots,\gb_m)\in(\Z/N\Z)^m$.
Then we have
\begin{equation*}
     t_N\cdot [\{1\}^m,\bfs;\bfgb',\bfga]
     =\sum_{\gam=0}^{N-1} \Big[z_{1,\gam} * \big(z_{1,\gb_1}\dots z_{1,\gb_m} z_{s_1;\ga_1}\dots z_{s_d;\ga_d} \big)  \Big].
\end{equation*}
Thus by induction
\begin{equation*}
\sum_{\ell=0}^{m}\sum_{\gam=0}^{N-1} [\{1\}^{m+1},\bfs;\gb_1,\dots,\gb_\ell,\gam,\gb_{\ell+1},\dots,\gb_m,\bfga]
\in \qMZ_N[t_N].
\end{equation*}
On the other hand
\begin{equation*}
\sum_{\gam=0}^{N-1} g_\gam(\{1\}^m,\bfs;\bfgb',\bfga) =N[\{1\}^{m+1},\bfs;0,\bfgb',\bfga]
-\sum_{\gam=0}^{N-1} [\{1\}^{m+1},\bfs;\gam,\bfgb',\bfga].
\end{equation*}
Adding up the above two equations we get the following element in $\qMZ_N[t_N]$
\begin{equation*}
N[\{1\}^{m+1},\bfs;0,\bfgb',\bfga]
+\sum_{\ell=1}^{m}\sum_{\gam=0}^{N-1} [\{1\}^{m+1},\bfs;
    \gb_1,\dots,\gb_\ell,\gam,\gb_{\ell+1},\dots,\gb_m,\bfga].
\end{equation*}
Subtracting $N g_{\gb_0}(\{1\}^m,\bfs;\bfgb',\bfga)$ we see that for all
$\bfgb=(\gb_0,\dots,\gb_m)\in(\Z/N\Z)^{m+1}$
\begin{equation*}
N[\{1\}^{m+1},\bfs;\bfgb,\bfga]
+\sum_{\ell=1}^{m}\sum_{\gam=0}^{N-1} [\{1\}^{m+1},\bfs;\gb_1,\dots,\gb_\ell,\gam,\gb_{\ell+1},\dots,\gb_m,\bfga]
=b(\bfgb),
\end{equation*}
for some $b(\bfgb)\in \qMZ_N[t_N]$.
Regarding $[\{1\}^{m+1},\bfs;\bfgb,\bfga]$ as $N^{m+1}$ variables as $\bfgb$ varies
we now only need to show the corresponding $N^{m+1}\times N^{m+1}$ coefficient
matrix $M(N,m)$ is nonsingular, which has been proved by \eqref{equ:MatrixMNm}.

(ii) We only need to show that $t_N$ is transcendental over $\qMZ_N$. By definition
\begin{align*}
t_N=\sum_{\ga=0}^{N-1} [1;\ga]
=&\, \sum_{\ga=0}^{N-1} \sum_{n=1}^\infty \sum_{d|n} \eta^{d\ga} q^n \\
=&\, \sum_{n=1}^\infty \sum_{d|n} N q^{Nn}  \approx \frac{-N\log(1-q^N)}{1-q^N} \quad (\text{when $q$ is close to }\eta^{-1})
\end{align*}
by \cite[Lemma~2]{Pupyrev2005}. Now it is an easy computation to show that when $q$ is close to $\eta^{-1}$
$$ \frac{-N\log(1-q^N)}{1-q^N} \approx \frac{-\log(1-\eta q)}{1-\eta q}.$$
By \propref{prop:Zkexpli} we see that if $|\bfs|=k>1$
we have the approximations $[\bfs;\bfga]\approx \frac{1}{(1-\eta q)^k}$ when $q$ close to $\eta^{-1}$.
Also by \propref{prop:Zkexpli}, if $|\bfs|=k-1$ then
$$\lim_{q\to 1} (1-q)^k g_\gb(\bfs;\bfga)[q\eta^{-1}]=\gG_\gb(\bfs;\bfga)$$
which is finite by the definition \eqref{equ:gG_gb}, and therefore when $q$ is close to $\eta^{-1}$
$g_\gb(\bfs;\bfga)\approx \frac{1}{(1-\eta q)^k}.$
Consequently, $t_N$ is transcendental over $\qMZ_N[T]$.

We have completed the proof of the theorem.
\end{proof}

The following result is a simple generalization of \cite[Lemma 2.5]{BachmannKu2013}.
\begin{lem} \label{lem:eulerpol}
For $s_1,\dots,s_d \in \N$ and $\ga_1,\dots,\ga_d\in \Z/N\Z$ we have
\begin{equation*}
[s_1,\dots,s_d;\ga_1,\dots,\ga_d ](q) =
\frac{1}{(s_1-1)! \dots (s_l-1)!} \sum_{n_1 > \dots > n_d > 0} \prod_{j=1}^d
\frac{\eta^{\ga_j} q^{n_j} P_{s_j-1}\left( \eta^{\ga_j} q^{n_j} \right)}{(1-\eta^{\ga_j} q^{n_j})^{s_j}},
\end{equation*}
where $P_k(t)= \sum_{n=0}^{k-1} A_{k,n} t^n$ is the $k$-th Eulerian polynomial with
the coefficient $A_{k,n}$ defined by
\begin{equation*}
A_{k,n} = \sum_{i=0}^n (-1)^i \binom{k+1}{i} (n+1-i)^k,
\end{equation*}
which is positive.
\end{lem}

\begin{proof}
By \propref{prop:ExpressMDFbyLi} and the properties of Eulerian polynomial
(see \cite{OEIS,Foata2010}) we have
\begin{equation*}
\sum_{n_1 > \dots > n_d > 0} \prod_{j=1}^d
\frac{\eta^{\ga_j}q^{n_j} P_{s_j-1}\left( \eta^{\ga_j}q^{n_j} \right)}{(1-\eta^{\ga_j}q^{n_j})^{s_j}}
= \sum_{n_1>\dots>n_d>0}\prod_{j=1}^d\sum_{v_j=1}^\infty \eta^{v_j\ga_j} v_j^{s_j-1} q^{v_j n_j}
= \sum_{n=1}^\infty \gs_{\bfs-\bfone}^{\bfga}(n) q^n.
\end{equation*}
The positivity of $A_{k,n}$ follows from its combinatorial interpretation as
the number of permutations of $\{1,2,..,n\}$ with $k$ ascents (see \cite{OEIS}
or search the sequence A008292 on oeis.org). This is also proved by Pupyrev \cite[Lemma~1]{Pupyrev2005}.
The lemma follows at once.
\end{proof}

\begin{thm} \label{thm:Zk}
{\rm (i)} The $\Q_N$-linear map $Z: \qMZ_N \to \MZ_N$ defined by
 \begin{align*}
 Z([\bfs;\bfga]) &\, = \zeta_{N}(\bfs;\bfga) \quad \forall s_1>1,\\
 Z(g_\gb(\bfs;\bfga)) &\, =\gG_\gb(\bfs;\bfga) \quad \forall \gb\in\Z/N\Z
\end{align*}
is a homomorphism of weight-filtered algebras. It can be factored through the weight-graded object
$(\ol{\qMZ_N},*)$:
\begin{equation*}
\xymatrix{\qMZ_N \ar[r]\ar[dr]_Z &  \overline{\qMZ_N} \ar[d]^{\ol{Z}} \\   & \MZ_N}
\end{equation*}
where the multiplication $*$ of $\overline{\qMZ_N}$
is inherited from that of $\qMZ_N$. Moreover, $(\ol{\qMZ_N},*)$
satisfies the stuffle relations and $\ol{Z}$ induces a $\Q$-algebra homomorphism
 $$\ol{Z}_\Q:(\ol{\qMZ_N}(\Q),*)\lra (\MZ_N(\Q),\cdot).$$

{\rm (ii)} Extending $Z$ by setting
  \begin{align*}
 Z([\bfs;\bfga]_N t_N^r) &= \zeta_{N}(\bfs;\bfga)T^r
\end{align*}
we then obtain a $\Q_N$-algebra homomorphisms
\begin{equation*}
\xymatrix{\MD_N = \qMZ_N[t_N] \ar[r]\ar[dr]_Z &  \overline{\qMZ_N} [t_N] \ar[d]^{\ol{Z}} \\   & \MZ_N[T],}
\end{equation*}
such that
\begin{equation*}
Z([1,s_2,\dots,s_d;\ga_1,\dots,\ga_d])=\zeta^\ast_{N}(1,s_2,\dots,s_d;\ga_1,\dots,\ga_d)
\end{equation*}
as defined in section \ref{sec:regularizationMZV}. Moreover, $\ol{Z}$ induces a $\Q$-algebra homomorphism
 $$\ol{Z}_\Q: \overline{\qMZ_N}(\Q)[t_N] \lra \MZ_N(\Q)[T].$$
\end{thm}
\begin{proof}
We only need to show the property on $\ol{Z}$. But this follows immediately from the fact that
the matrix $M(N,m)$ in the proof of \thmref{thm:polyadalg} has integer entries so that its inverse
has rational entries.
\end{proof}

For any $\bfs\in\N^d$  and $\bfga\in(\Z/N\Z)^d$ we define
\begin{alignat*}{3}
Z_k\big([\bfs;\bfga](q)\big) =&\, \lim_{q \to 1^-} (1-q)^{k} [\bfs;\bfga]\big(q\eta^{-1}\big)
  &  \quad & \forall |\bfs|\le k,\\
Z_k\big(g_\gb(\bfs;\bfga)\big) =&\, \lim_{q\to 1^-} (1-q)^k g_\gb(\bfs;\bfga)\big(q\eta^{-1}\big)
  &   \quad & \forall |\bfs|\le k-1.
\end{alignat*}
The next result shows that $Z_k$ is the restriction of $Z$ to $\filw_{k}(\qMZ_N)$,
\begin{prop}\label{prop:Zkexpli}
The map $Z_k$ is $\Q_N$-linear on $\filw_{k}(\qMZ)$ and
\begin{equation*}
Z_k\left( [\bfs;\bfga](q) \right) =
\left\{
  \begin{array}{ll}
   \zeta_N(\bfs;\bfga), & \hbox{if $s_1>1$ and $|\bfs|=k$;} \\
    \ \quad 0, & \hbox{if $|\bfs|<k$,}
  \end{array}
\right.
\end{equation*}
and for all $1\le \gb<N$
\begin{equation}\label{equ:Z_kg_gb}
Z_k\big(g_\gb(\bfs;\bfga)\big)
=
\left\{
  \begin{array}{ll}
   \gG_\gb(\bfs;\bfga), & \hbox{if $|\bfs|=k-1$;} \\
    \ \quad 0, & \hbox{if $|\bfs|< k-1$.}
  \end{array}
\right.
\end{equation}
\end{prop}
\begin{proof}
First we assume $s_1>1$. Using \lemref{lem:eulerpol} we get
{\allowdisplaybreaks
\begin{align} \notag
 &\ Z_k\big( [s_1,\dots, s_d;\ga_1,\dots,\ga_d ] \big)  \notag\\
 &= \lim_{q \to 1^-}\big((1-q)^k [s_1,\dots, s_d;\ga_1,\dots,\ga_d] (q\eta^{-1}) \big) \notag \\
 &=\lim_{q \to 1^-}\left(
   (1-q)^k \sum_{n_1 > \dots > n_d > 0} \prod_{j=1}^d
\frac{ \eta^{\ga_j-n_j} q^{n_j} P_{s_j-1}\big( \eta^{\ga_j-n_j} q^{n_j} \big) }
{(s_j-1)! \big(1-\eta^{\ga_j-n_j} q^{n_j}  \big)^{s_j} } \right)  \label{equ:ConvSeries}\\
 &= \sum_{n_1 > \dots > n_d > 0}\prod_{j=1}^d \lim_{q \to 1^-}\left(
   (1-q)^{s_j}
\frac{ \eta^{\ga_j-n_j} q^{n_j} P_{s_j-1}\big( \eta^{\ga_j-n_j} q^{n_j} \big) }
{(s_j-1)! \big(1-\eta^{\ga_j-n_j} q^{n_j}  \big)^{s_j} } \right) \notag \\
 &= \sum_{\substack{n_1 > \dots > n_d > 0 \\ n_j\equiv \ga_j\ppmod{N}\ \forall 1\le j\le d} }
\prod_{j=1}^d \lim_{q \to 1^-}\left(
   (1-q)^{s_j}\frac{q^{n_j} P_{s_j-1}\big( q^{n_j} \big) }
{(s_j-1)! \big(1-q^{n_j}\big)^{s_j}  } \right) \notag \\
  &  = \zeta_N(s_1,\dots,s_d;\ga_1,\dots,\ga_d). \notag
 \end{align}}
Here we have used the fact that the $k$-th Eulerian polynomial $P_k(t)$ satisfies $P_k(1)=k!$.
We also need to justify the exchange of the order of taking the limit and taking the infinite sum. But
by triangle inequality
$$1\le |1-\eta^{\ga-n} q^n|+|\eta^{\ga-n} q^n|=|1-\eta^{\ga-n} q^n|+q^n.$$
Thus for all $q\in[1/2, 1)$ we have
$$\frac1{|1-\eta^{\ga-n} q^n|}\le \frac1{1-q^n}.$$
Moreover, since $P_{s_j-1}(x)$ has positive coefficients by \lemref{lem:eulerpol}
we see that
$$|P_{s_j-1}(\eta^{\ga_j-n_j} q^n)|\le P_{s_j-1}(q^{n_j}) \quad\text{for all $j$}.$$
Hence the series
appearing in \eqref{equ:ConvSeries} after the limit symbol converges absolutely:
\begin{align*}
&\,  (1-q)^k \sum_{n_1 > \dots > n_d > 0} \prod_{j=1}^d
\left|\frac{ \eta^{\ga_j-n_j} q^{n_j} P_{s_j-1}\big( \eta^{\ga_j-n_j} q^{n_j} \big) }
{(s_j-1)! \big(1-\eta^{\ga_j-n_j} q^{n_j}  \big)^{s_j} } \right|\\
\le &\, \sum_{n_1 > \dots > n_d > 0} \prod_{j=1}^d
 \frac{ (1-q)^{s_j}q^{n_j} P_{s_j-1}(q^{n_j}) }{(s_j-1)!(1-q^{n_j})^{s_j} }
\le  \zeta(2,1,\dots,1)
\end{align*}
by \cite[Lemma~6.6(ii)]{BachmannKu2013}.

We now turn to the proof of \eqref{equ:Z_kg_gb}. By definition \eqref{equ:defng_gb1}, if $1\le \gb<\ga_1$ then
\begin{align}
 &\,  Z_k\left( g_\gb(\bfs;\bfga) \right)=\lim_{q \to 1^-} \sum_{n>n_1 > \dots > n_d > 0}  \notag\\
& (1-q)^k \left( \frac{\eta^{-n} q^n }{1-\eta^{-n} q^n }-\frac{\eta^{-n} q^{\gb+n}}{1-\eta^{-n} q^{\gb+n} } \right)
\prod_{j=1}^d \frac{ \eta^{\ga_j-n_j} q^{n_j} P_{s_j-1} \big( \eta^{\ga_j-n_j} q^{n_j} \big) }
{(s_j-1)! \big(1-\eta^{\ga_j-n_j} q^{n_j}  \big)^{s_j} }   \label{equ:Z_kg_gb11}\\
 &\, -\lim_{q \to 1^-} (1-q)^k \sum_{\ell=1}^\gb \sum_{n_1>\dots>n_d>0}
\frac{\eta^{\gb-n_1-\ell} q^{n_1+\ell} } {1-\eta^{\gb-n_1-\ell} q^{n_1+\ell} }
\prod_{j=1}^d \frac{ \eta^{\ga_j-n_j} q^{n_j} P_{s_j-1} \big( \eta^{\ga_j-n_j} q^{n_j} \big) }
{(s_j-1)! \big(1-\eta^{\ga_j-n_j} q^{n_j}  \big)^{s_j} }. \label{equ:Z_kg_gb12}
\end{align}
It's easy to see that
\eqref{equ:Z_kg_gb12} vanishes unless $d>0$ and there exist $n_1$ and $\ell$ such that
\begin{equation*}
1\le \ell\le\gb,\quad  \ga_1\equiv n_1 \quad\text{and}\quad \gb\equiv n_1+\ell \pmod{N}.
\end{equation*}
But this forces $\ell+\ga_1\equiv \gb \pmod{N}$ which is impossible since $\gb<\ell+\ga_1<\gb+N$.
Hence \eqref{equ:Z_kg_gb12}=0.
On the other hand, by the proof of \cite[Lemma 6.6(ii)]{BachmannKu2013} we have
\begin{equation*}
    h_n(q):=\frac{(1-q)^2 q^n }{(1-q^{n})^2}-\frac{1}{n^2}<0 \quad \forall q\in[1/2, 1).
\end{equation*}
Therefore
\begin{equation*}
    (1-q)\left|\frac{\eta^{-n} q^{n}}{1-\eta^{-n} q^{n} }
-\frac{\eta^{-n} q^{\gb+n}}{1-\eta^{-n} q^{\gb+n}} \right|
=\left|\frac{(1-q)(1-q^\gb) q^n}{(1-\eta^{-n} q^{n})(1-\eta^{\gb-n} q^n) }\right|
\le \frac{\gb (1-q)^2 q^n }{(1-q^{n})^2} <\frac{\gb}{n^2}.
\end{equation*}
Hence by exchanging the order
of the limit and the summation \eqref{equ:Z_kg_gb11} we arrive at \eqref{equ:Z_kg_gb}.

When $\gb\ge \ga_1$ we can use definition \eqref{equ:defng_gb2} and prove the claim of
the proposition in a similar way. We only point out that
\begin{equation*}
\left|\frac{1-q^{\gb-N}}{1-q}\right|<2^{N-\gb}(N-\gb)  \quad \forall q\in[1/2, 1)
\end{equation*}
and leave the rest of the details to the interested reader.
\end{proof}

\section{A derivation and linear relations in $\MD_N$}\label{sec:derivation}
Our  main result in this section is that the $\Q_N$-linear operator $\dif=q \frac{d}{dq}$ on $\MD_N$
is a derivation and it increases the weight by at most 2 and the depth by at most 1 (see \thmref{thm:derivative}).
Although this is completely similar to the level 1 case proved
in \cite{BachmannKu2013}, there is some technical complications caused by
the appearance of the colors. Our proof is constructive, similar to the level $N=1$ case,
but we will derive the explicit formula for all depths. Since
the depth one case is a little different from the general case so it will be treated separately.
We will provide some concrete examples at the end of this section.

\begin{lem}\label{lem:TxTy}
The generating series $T_{\ga_1,\dots,\ga_d}(x_1,\dots,x_d)$ of MDFs of
depth $d$ can be written as
\begin{align*}
T_{\ga_1,\dots,\ga_d}(x_1,\dots,x_d) =&\, \sum_{s_1,\dots,s_d >0}
    [s_1,\dots,s_d;\ga_1,\dots,\ga_d] x_1^{s_1-1} \dots x_d^{s_d-1} \\
=&\, \sum_{n_1,\dots,n_d>0} \prod_{j=1}^d
\frac{e^{n_j x_j} \eta^{\ga_j n_j} q^{n_1+\dots+n_j}}{1- q^{n_1+\dots+n_j}}.
\end{align*}
\end{lem}
\begin{proof}
This can be seen by direct computation. See \cite[Lemma~3.1]{BachmannKu2013} for details.
\end{proof}

A prototype of the following can be found also in \cite{GKZ2006}.

\begin{lem} \label{lem:shuffle2}
We have
\begin{align*}
T_\ga(x)\cdot T_\gb(y) &\, = T_{\ga+\gb,\ga}(x+y,x)+T_{\ga+\gb,\gb}(x+y,y)\\
 &\,  - T_{\ga+\gb}(x+y) + R_{\ga+\gb}^{(1;1)}(x+y),
\end{align*}
where
\begin{equation*}
R_\ga^{(1;1)}(x) =\sum_{n>0} e^{nx} \frac{\eta^{\ga n} q^n}{(1-q^n)^2}
= \sum_{k>0} \frac{\dif[k;\ga]}{k} x^k+[2;\ga].
\end{equation*}
\end{lem}

\begin{proof} By \lemref{lem:shuffle2}
\begin{align*}
T_\ga(x) T_\gb(y) &\, = \sum_{n_1,n_2 > 0}\eta^{\ga n_1+\gb n_2}  e^{n_1 x + n_2 y} \qe{n_1} \qe{n_2} \\
&\, = \sum_{n_1>n_2>0}\dots+ \sum_{n_2 > n_1>0} \dots +  \sum_{n_2 = n_1>0} \dots =: F_1 + F_2 + F_3,
\end{align*}
where by using the substitution $n_1 = n_2 + n_1'$
\begin{align*}
F_1 = &\sum_{n_1>n_2>0} \eta^{\ga n_1+\gb n_2}  e^{n_1 x + n_2 y} \qe{n_1} \qe{n_2} \\
=& \sum_{n_1',n_2>0}  \eta^{\ga n_1'+(\ga+\gb)n_2} e^{n_1' x + n_2 (x+y)}\qe{n_2}  \qe{n_1'+n_2}
= T_{\ga+\gb,\ga}(x+y,x).
\end{align*}
By similar argument $F_2= T_{\ga+\gb,\gb}(x+y,y).$
Now, since $\big( \qe{n} \big)^2 = \frac{q^n}{(1-q^n)^2} - \frac{q^n}{1-q^n}$
\begin{align*}
F_3 = &\sum_{n>0}  \eta^{\ga n+\gb n} e^{nx +ny} \left( \frac{q^n}{1-q^n} \right)^2 \\
= &\sum_{n>0} \eta^{(\ga+\gb) n}e^{n(x+y)}\frac{q^n}{(1-q^n)^2}
-\sum_{n>0} \eta^{(\ga+\gb) n}e^{n (x+y)}\frac{q^n}{1-q^n} \\
= &\, R_{\ga+\gb}^{(1;1)}(x+y)-T_{\ga+\gb}(x+y).
\end{align*}

To express $R_\ga^{(1;1)}(x)$ using the derivation $\dif$ we notice that
\begin{equation*}
 \sum_{k>0} \dif[k;\ga] x^{k-1}
= \dif T_\ga(x) =\dif \sum_{n>0} \eta^{ \ga n} e^{nx} \qe{n} = \sum_{n>0} n  \eta^{ \ga n} e^{nx} \frac{q^n}{(1-q^n)^2}.
\end{equation*}
Thus
\begin{align*}
 \sum_{k>0} \frac{\dif [k;\ga]}{k} x^{k}
&=  \int_0^x  \dif T_\ga(t) dt = \sum_{n>0} \int_0^x n \eta^{ \ga n} e^{nt} dt \frac{q^n}{(1-q^n)^2} \\
&= \sum_{n>0} \eta^{ \ga n} e^{nx} \frac{q^n}{(1-q^n)^2} - [2;\ga].
\end{align*}
This completes the proof of the lemma.
\end{proof}

\begin{prop} \label{prop:formularDifd=2}
For $\ga,\gb\in\Z/N\Z$, $s_1,s_2\in\N$ with $s_1+s_2>2$ and $s=s_1+s_2-2$
we have
\begin{align*}
\binom{s}{s_1-1} \frac{\dif[s;\ga+\gb]}{s} &\, =
[s_1;\ga]\cdot [s_2;\gb] +\binom{s}{s_1-1} [s+1;\ga+\gb] \\
&\, - \sum_{a+b=s+2} \left( \binom{a-1}{s_1-1}[a,b;\ga+\gb,\gb]
    +\binom{a-1}{s_2-1}[a,b;\ga+\gb,\ga] \right).
\end{align*}
\end{prop}
\begin{proof}
We have
\begin{align*}
T_{\ga+\gb,\ga}(x+y,x) &\,=\sum_{s_1,s_2 >0}\sum_{a=s_2}^{s_1+s_2-1}
    \binom{a-1}{s_2-1} [a,s_1+s_2-a;\ga+\gb,\ga] x^{s_1-1} y^{s_2-1} ,\\
T_{\ga+\gb,\gb}(x+y,y) &\,=\sum_{s_1,s_2 >0}\sum_{a=s_1}^{s_1+s_2-1}
    \binom{a-1}{s_1-1} [a,s_1+s_2-a;\ga+\gb,\gb] x^{s_1-1} y^{s_2-1} ,\\
T_{\ga+\gb}(x+y) &\, = \sum_{s_1,s_2 >0}  \binom{s_1+s_2-2}{s_1-1}
    [s_1+s_2-1;\ga+\gb]  x^{s_1-1} y^{s_2-1} , \\
\sum_{k>0} \frac{\dif[k;\ga+\gb]}{k} (x+y)^{k} &\, = \sum_{s_1,s_2 >0}
\binom{s_1+s_2-2}{s_1-1} \frac{\dif[s_1+s_2-2;\ga+\gb]}{s_1+s_2-2}  x^{s_1-1} y^{s_2-1}.
\end{align*}
Now the proposition follows immediately from \lemref{lem:shuffle2}.
\end{proof}

\begin{eg}\label{eg:Dif[1;1]}
Taking $\ga=0$ and $\gb=1$ we get for all $N\ge 0$
\begin{align*}
\dif[1;1]=&\, [2;0]\cdot [1;1]+[2;1]-[2,1;1,1]-[2,1;1,0]-[1,2;1,0]  \\
=&\, [2,1;0,1]-[2,1;1,1]-[2,1;1,0]+[2;1]
-\om^N_{0;1} [2;0]
+\om^N_{1;1} \big([1;0]-[1;1] \big)
\end{align*}
by \propref{prop:l2-expli} and \eqref{equ:gl=om}. Similarly
\begin{align*}
\dif[2;1]=&\, [2;0]\cdot [2;1]+2[3;1]-[2,2;1,0]-2[3,1;1,0]-[2,2;1,1]-2[3,1;1,1] \\
=&\, [2,2;0,1]-[2,2;1,1]-2[3,1;1,0]-2[3,1;1,1]+2[3;1]\\
-&\,\om^N_{1;1} \big([2;0]+[2;1] \big) +\om^N_{2;1} \big([1;0]-[1;1] \big),
\end{align*}
again by \propref{prop:l2-expli}  and \eqref{equ:gl=om}.
\end{eg}

The next result provides us with a way to find $\Q$-linear relations in $\MZV(w,N)$ systematically.
\begin{thm}\label{thm:ZD}
For any positive integer $w\ge 3$ and $N$ we have
\begin{equation*}
  \filw_{w-1}(\MD_N) \subseteq  \ker(Z_w),\quad\text{and}\quad
    \dif \filw_{w-2}(\MD_N) \subseteq  \ker(Z_w).
\end{equation*}
\end{thm}
\begin{proof}
Essentially the same proof for the level $N=1$ case used by Bachman and K\"uhn  
(see \cite[Proposition~7.3]{BachmannKu2013}) works in general. We leave the details to
the interested reader.
\end{proof}

\begin{eg}\label{eg:relderiva}
Suppose $N=2$. By Example~\ref{eg:Dif[1;1]}
 \begin{align}
\dif [\baro{1}]=&\, [2,\baro{1}]-[\baro{2},1]-[\baro{2},\baro{1}]+[\baro{2}]+\frac14\Big([1]-[\baro{1}]-2[2]\Big),\label{equ:Dodd1}\\
\dif [\baro{2}]=&\, 2[\baro{3}]-2[\baro{3},\baro{1}]-[\baro{2},\baro{2}]
    -2[\baro{3},1]+[2,\baro{2}]-\frac14\Big([2]+[\baro{2}]\Big).\label{equ:Dodd2}
\end{align}
Applying the map $Z$ to \eqref{equ:Dodd1} and \eqref{equ:Dodd2} we get two relations
at level 2 for the MZVs of weight 3 and 4, respectively:
\begin{align}
 \zeta_{2}(2,\baro{1})=&\, \zeta_{2}(\baro{2},1)+\zeta_{2}(\baro{2},\baro{1}) \label{equ:N=2wt3Reuse}\\
 \zeta_{2}(2,\baro{2})=&\, 2\zeta_{2}(\baro{3},\baro{1})+\zeta_{2}(\baro{2},\baro{2})+2\zeta_{2}(\baro{3},1).\notag
\end{align}
Notice that
{\allowdisplaybreaks
\begin{align}
4\zeta_{2}(2,\baro{1})=&\, \zeta(2,1)-\zeta(2,\bar1)+\zeta(\bar2,1)-\zeta(\bar2,\bar1),\notag\\
4\zeta_{2}(\baro{2},1)=&\, \zeta(2,1)-\zeta(\bar2,1)+\zeta(2,\bar1)-\zeta(\bar2,\bar1),\notag\\
4\zeta_{2}(\baro{2},\baro{1})=&\, \zeta(2,1)-\zeta(2,\bar1)-\zeta(\bar2,1)+\zeta(\bar2,\bar1),\notag\\
4\zeta_{2}(\baro{3},1)=&\, \zeta(3,1)-\zeta(\bar3,1)+\zeta(3,\bar1)-\zeta(\bar3,\bar1), \label{equ:N=2Wt=41} \\
4\zeta_{2}(\baro{3}, \baro{1})=&\, \zeta(3,1)-\zeta(3,\bar1)-\zeta(\bar3,1)+\zeta(\bar3,\bar1),\label{equ:N=2Wt=42} \\
4\zeta_{2}(2,\baro{2})=&\, \zeta(2,2)-\zeta(2,\bar2)+\zeta(\bar2,2)-\zeta(\bar2,\bar2),\label{equ:N=2Wt=43} \\
4\zeta_{2}(\baro{2},\baro{2})=&\,   \zeta(2,2)-\zeta(2,\bar2)-\zeta(\bar2,2)+\zeta(\bar2,\bar2).\label{equ:N=2Wt=44}
\end{align}}
Therefore we get the relations among alternating Euler sums:
\begin{align*}
3\zeta(\bar2,1)=&\, \zeta(2,1)+\zeta(2,\bar1)+\zeta(\bar2,\bar1), \\
2\zeta(3,1)=&\, 2\zeta(\bar3,1)+\zeta(\bar2,2)-\zeta(\bar2,\bar2).
\end{align*}
These two relations can be easily verified using the explicit structural results
contained in \cite[Proposition 3.3 and 3.5]{Zhao2007unpub}
(unfortunately, these were omitted in the published version \cite{Zhao2010a}).
\end{eg}

To deal with the general depth we define for all $\ga_1,\dots,\ga_d\in(\Z/N\Z)^d$ and $1\le j\le d$
\begin{equation*}
R_{\ga_1,\dots,\ga_d}^{(j;d)}(x_1,\dots,x_d)
 :=\sum_{n_1,\dots,n_d >0}  \prod_{i=1}^d
\frac{\eta^{\ga_in_i} e^{n_i x_i } q^{n_1+\cdots+n_i}}{(1-q^{n_1+\dots+n_i})^{\delta_{i,j}+1}}.
\end{equation*}
and for $\bfy=(y_1,\dots,y_d)$
\begin{equation}\label{equ:Rdefn}
 R_{\gb;\bfga}^{(d)}(x;\bfy)
 =\sum_{j=1}^d  R_{\gb+\ga_1,\dots,\gb+\ga_j,\ga_{j+1},\dots,\ga_d}^{(j;d)}(x+y_1,\dots,x+y_j,y_{j+1},\dots,y_d).
\end{equation}
The operator $D_x$ on functions $f(x,y,\cdots)$ is defined by
\begin{equation*}
D(f) =  \frac{\partial f(x,y,\cdots)}{\partial x}   \Big|_{x=0}.
\end{equation*}

Observe that $D_x(R_{\gb;\ga}^{(1)}(x;y)) = D_x(R_{\ga+\gb}^{(1;1)}(x+y)) = \dif T_{\ga+\gb}(y)$.
This can be partially generalized to arbitrary depths by the following lemma.
\begin{lem} \label{lem:AllDepth}
For all $d\ge 2$ and $\ga_1,\dots,\ga_d\in(\Z/N\Z)^d$ we have
\begin{align*}
D_x\Big( R_{0;\ga_1,\dots,\ga_d}^{(d)}(x;y_1,\dots,y_d) \Big)
= \dif T_{\ga_1,\dots,\ga_d}(y_1,\dots,y_d).
\end{align*}
\end{lem}
\begin{proof}
For any fixed $j$ we have
\begin{multline*}
D_x\Big(R_{\ga_1,\dots,\ga_d}^{(j;d)}(x+y_1,\dots,x+y_j,y_{j+1},\dots,y_d)\Big) = \\
 \sum_{n_1,\dots,n_d >0} (n_1+\dots+n_j)
\eta^{\ga_1n_1+\dots+\ga_d n_d}  e^{n_1 y_1 + \dots + n_d y_d} \prod_{i=1}^d \frac{q^{n_1+\dots+n_i}}{(1+q^{n_1+\dots+n_i})^{\delta_{i,j}+1}}.
\end{multline*}
Adding all these equations together when $j$ goes from 1 to $d$ we now can prove the lemma easily
by using $\dif \qe{n} = \frac{ n \cdot q^n}{(1-q^n)^2}$ and the product formula.
\end{proof}

\begin{thm} \label{thm:derivative}
\emph{(i)} The operator $\dif = q \frac{d}{dq}$ is a derivation on $\MD_N$.

\emph{(ii)} Let $d\ge 2$, $\bfs=(s_1,\dots,s_d)\in\N^d$ and $\bfga=(\ga_1,\dots,\ga_d)\in(\Z/N\Z)^d$.
For any $j\le d$ let $\bfe_j=(0,\dots,1,\dots,0)\in \Z^d$ be the $j$-th standard unit vector and put
$\bfga(j)=(\ga_1,\dots,\ga_j,\ga_j,\dots,\ga_d)$ with only $\ga_j$ repeated once.
Then we have
\begin{align*}
\dif[\bfs;\bfga]=&[2;0]\cdot [\bfs;\bfga]
+\sum_{j=1}^d (d-j+1) s_j[\bfs+\bfe_j;\bfga]
-\sum_{j=1}^d s_j[\bfs+\bfe_j,1;\bfga,0] \\
-&\, [\bfs,2;\bfga,0]-\sum_{j=1}^d  \sum_{a+b=s_j+2} (a-1) [s_1,\dots,s_{j-1},a,b,s_{j+1},\dots,s_d;\bfga(j)] \\
-&\, \sum_{j=1}^d \sum_{\ell=1}^{j-1} \sum_{a+b=s_j+1} s_\ell [s_1+\gd_{\ell,1},\dots,s_{j-1}+\gd_{\ell,j-1},a,b,s_{j+1},\dots,s_d;\bfga(j)].
\end{align*}
Hence
$$\dif:\filwle_{k,d}(\MD_N)\lra \filwle_{k+2,d+1}(\MD_N).$$
\end{thm}
\begin{proof}
(i) Clearly $\dif$ is a derivation since for all $m,n\in \N$ we have
$\dif (q^{m+n})= (m+n)q^{m+n}=\dif (q^m)q^n+q^m\dif (q^n)$.

(ii) Put $\bfy=(y_1,\dots,y_d)$ and $x+\bfy=(x+y_1,\dots,x+y_d)$. We have
 \begin{equation} \label{eq:deriveq}
 \begin{split}
 T_0(x) \cdot T_\bfga (\bfy)
& =\sum_{m,n_1,\dots,n_d>0} e^{mx + n_1y_1 + \dots + n_d y_d} \qe{m} \qe{n_1} \dots \qe{n_1+\dots+n_d} \\
& =T_{\bfga,0} (x+\bfy,x)
 + \sum_{j=1}^d T_{\bfga(j)} (x+y_1,\dots,x+y_j,y_j,\dots,y_d)   \\
&+ R_{0;\bfga}^{(d)}(x;\bfy)
 - \sum_{j=1}^d  T_\bfga(x+y_1,\dots,x+y_j,y_{j+1},\dots,y_d).
\end{split}
\end{equation}
We obtain the above expansion by breaking up the sum in the first line
in the following way: second line is from
the sums with $n_1 + \dots + n_d < m$ and $n_1 + \dots + n_{j-1} < m < n_1 + \dots + n_j$
for $j=1,\dots,d+1$ . Setting $m = n_1 + \dots + n_{j-1} + m'$
and $n_j = m' + n_j'$ for these terms it is easy to see that one gets the sum over
$m',n_1,\dots,n_j',\dots,n_d$ which then gives $T_{\bfga(j)}(x+y_1,\dots,x+y_j,y_j,\dots,y_d)$
when $j\le d$ and $T_{\bfga,0} (x+\bfy,x)$ when $j=d+1$.
The third line of \eqref{eq:deriveq} arises from the sum over $m = n_1 + \dots + n_j$.
In this case one uses the definition
\eqref{equ:Rdefn} together with the identity
\begin{equation*}
\left( \qe{n} \right)^2 = \frac{q^n}{(1-q^n)^2} - \frac{q^n}{1-q^n}.
\end{equation*}

Finally, applying $D_x$ on both sides of \eqref{eq:deriveq} and
comparing the coefficient of $y_1^{s_1-1}\dots y_d^{s_d-1}$
one can derive the expression of $\dif$ immediately from \lemref{lem:AllDepth}.
This finishes the proof of the theorem.
\end{proof}

\begin{cor} \label{cor:depth2}
For all $s,t\in\N$ and $\ga,\gb\in\Z/N\Z$ we have
\begin{align*}
\dif&[s,t;\ga,\gb] = [2;0]\cdot [s,t;\ga,\gb]+ 2 s [s+1,t;\ga,\gb] + t [s,t+1;\ga,\gb]- [s,t,2;\ga,\gb,0] \\
&- s [s+1,t,1;\ga,\gb,0] - t [s, t+1,1;\ga,\gb,0] -\sum_{a+b=s+2} (a-1) [a,b,t;\ga,\ga,\gb] \\
&- \sum_{a+b=t+2} (a-1) [s,a,b;\ga,\gb,\gb]-\sum_{a+b=t+1} s [s+1,a,b;\ga,\gb,\gb] .
\end{align*}
\end{cor}

\begin{eg}\label{eg:D1111}
Suppose $N\ge 2$.
Using  \corref{cor:depth2} and the relation obtained from Example~\ref{eg:stuffle[a;ga][b,c;gb,gam]}
\begin{multline*}
\dif[1,1;1,1]=[2,1,1;0,1,1]-[2,1,1;1,1,0]-2[2,1,1;1,1,1]+[1,2,1;1,0,1]\\
-[1,2,1;1,1,0]-[1,2,1;1,1,1]+2[2,1;1,1]+[1,2;1,1]  \\
+  \om^N_{1;1} \big([1,1;1,0]+ [1,1;0,1]-2 [1,1;1,1]\big)-\om^N_{0;1} \big([1,2;1,0]+[2,1;0,1]\big)
\end{multline*}
by \propref{prop:l2-expli} and \eqref{equ:gl=om},
we see that by \thmref{thm:ZD} the image of the above under the map $Z_4$ is
\begin{multline*}
 \zeta_N(2,1,1;0,1,1)-\zeta_N(2,1,1;1,1,0)-2\zeta_N(2,1,1;1,1,1)\\
=\zeta^*_N(1,2,1;1,1,0)+\zeta^*_N(1,2,1;1,1,1)-\zeta^*_N(1,2,1;1,0,1).
\end{multline*}
By stuffle relations we must have
\begin{multline}\label{equ:N=2wt4stf1}
\zeta_N(2,1,1;1,0,1)-\zeta_N(2,1,1;0,1,1)=\zeta_N(2,2;0,1)-\zeta_N(2,2;1,1)\\
-\zeta_N(3,1;1,0)-\zeta_N(3,1;1,1)
\end{multline}
and
\begin{equation}\label{equ:N=2wt4stf2}
\zeta_N(2,1;1,0)+\zeta_N(2,1;1,1)=\zeta_N(2,1;0,1).
\end{equation}
When $N=2$ \eqref{equ:N=2wt4stf2} is just \eqref{equ:N=2wt3Reuse} while
\eqref{equ:N=2wt4stf1} also follows from \cite[Proposition~3.5]{Zhao2007unpub}, \eqref{equ:N=2Wt=41} to \eqref{equ:N=2Wt=44} and the following identities:
{\allowdisplaybreaks
\begin{align}
\begin{split}
 8\zeta_2(\baro{2},1,\baro{1})=&\, \zeta(2,1,1)-\zeta(\bar2,1,1)+\zeta(2,\bar1,1)-\zeta(\bar2,\bar1,1)\\
-&\, \zeta(2,1,\bar1)+\zeta(\bar2,1,\bar1)-\zeta(2,\bar1,\bar1)+\zeta(\bar2,\bar1,\bar1),
\end{split}
\label{equ:N=2Wt=45}\\
\begin{split}
 8\zeta_2(2,\baro{1},\baro{1})=&\, \zeta(2,1,1)+\zeta(\bar2,1,1)-\zeta(2,\bar1,1)-\zeta(\bar2,\bar1,1)\\
-&\, \zeta(2,1,\bar1)-\zeta(\bar2,1,\bar1)+\zeta(2,\bar1,\bar1)+\zeta(\bar2,\bar1,\bar1).
\end{split} \label{equ:N=2Wt=46}
\end{align}}
\end{eg}

In the next example we show how one can apply the Leibniz rule of $\dif$
to produce some $\Q$-linear relations in $\MD_N$.
\begin{eg}\label{eg:Lib}
Suppose $N\ge 2$.
Applying $\dif$ to \eqref{equ:[1;1][1;1]} one gets for all $N$
\begin{equation*}
\dif([1;1]\cdot [1;1])=2\dif[1,1;1,1]+\dif[2;1]-\dif[1;1]=2[1;1]\cdot\dif[1;1].
\end{equation*}
Using Examples~\ref{eg:stuffle[a;ga][b,c;gb,gam]}, \ref{eg:Dif[1;1]}, \ref{eg:D1111}  and
\propref{prop:l2-expli} one can finally gets a $\Q$-linear relation in $\MD_N$:
\begin{align*}
&\, 2[2,1,1;0,1,1]-2[2,1,1;1,0,1]=[2,2;1,1]-[2,2;0,1] +[2,1;0,1]\\
+&\,\om^N_{1;1}\big(2[1,1;1,1]-2[1,1;0,1]+[2;1]-[2;0]-[1;0]-[1;1]\big)-2(\om^N_{0;1})^2[2;0] \\
+&\, \om^N_{0;1}\big([2;0]+2[2,1;0,1]+2[2,1;1,1]-2[2,1;1,0]\big)+\big(\om^N_{2;1}
+4\om^N_{0;1}\om^N_{1;1}\big)\big([1;0]-[1;1]\big).
\end{align*}
Applying $Z_4$ we get
\begin{equation*}
2\zeta_N(2,1,1;0,1,1)-2\zeta_N(2,1,1;1,0,1)=\zeta_N(2,2;1,1)-\zeta_N(2,2;0,1).
\end{equation*}
When $N=2$ this can be proved also by \eqref{equ:N=2Wt=43}, \eqref{equ:N=2Wt=44}, \eqref{equ:N=2Wt=45},
\eqref{equ:N=2Wt=46}, and \cite[Proposition~3.5]{Zhao2007unpub}.
\end{eg}

\end{document}